\documentclass[11pt,reqno]{amsart}
\usepackage[utf8]{inputenc}
\usepackage{xcolor}
\usepackage[T1]{fontenc}                        
\usepackage{geometry}

\usepackage{amsfonts}
\usepackage{comment}                              
\usepackage{mparhack} 
\usepackage{amssymb}      

\usepackage{enumerate}
\usepackage{enumitem}
\usepackage{booktabs}                             
\usepackage{graphicx,subfig}                    
\usepackage{wrapfig}                            
\usepackage[bookmarks=true,colorlinks=true]{hyperref}\hypersetup{urlcolor=blue,citecolor=red,linkcolor=black}
\usepackage{bbm}
\usepackage[foot]{amsaddr}

\usepackage{bm}                                   
\usepackage{dsfont}

\newtheorem{theorem}{Theorem}[section]

\newtheorem{proposition}[theorem]{Proposition}
\newtheorem{lemma}[theorem]{Lemma}
\newtheorem{corollary}[theorem]{Corollary}
\theoremstyle{definition}

\newtheorem{remark}[theorem]{Remark}

\newtheorem{hypothesis}[theorem]{Hypothesis}

\let\originalleft\left
\let\originalright\right
\renewcommand{\left}{\mathopen{}\mathclose\bgroup\originalleft}
\renewcommand{\right}{\aftergroup\egroup\originalright}

\DeclareMathOperator*{\argmin}{argmin}
\newcommand{\me}{\mathcal{E}}

\newcommand{\mh}{\mathcal{H}}
\newcommand{\ml}{\mathcal{L}}

\newcommand{\mi}{\mathcal{I}}

\newcommand{\mr}{\mathcal{R}}
\newcommand{\ms}{\mathcal{S}}

\newcommand{\eepstau}{\me_{\eps, \tau}}
\newcommand{\eepstaurhat}[1]{\me_{\eps, \tau}(#1|\hat\rho)}

\newcommand{\dst}{\mathrm{d}}
\newcommand{\wass}{{\mathbf W}}

\newcommand{\mpt}{{\mathcal{P}_2}}

\newcommand{\mom}{{\mathfrak m}}

\newcommand{\R}{{\mathbb R}}
\newcommand{\N}{{\mathbb N}}
\newcommand{\Q}{{\mathbb Q}}
\newcommand{\Rp}{{\mathbb R}_{>0}}
\newcommand{\Rnn}{{\mathbb R}_{\ge0}}

\newcommand{\eps}{\varepsilon}

\newcommand{\intd}{\mathrm{d}}

\newcommand{\weakto}{\rightharpoonup}

\title[Global exponential stability in thin-film equation]{Global exponential stability of stationary profiles in a thin film equation with second-order diffusion}
\author{Christian Parsch$^\ast$}
\address{$^\ast$Technical University of Munich, Germany; TUM School of Computation, Information and Technology, Department of Mathematics, Boltzmannstrasse 3, D-80538 Garching}
\date{May 2025}

\begin{document}

\email{christian.parsch@tum.de}

\begin{abstract}
    We study existence and long-time behavior of weak solutions to a thin-film equation with a confinement potential and a second-order degenerate diffusion term. It is known that in absence of second order effects, solutions for general initial data converge at an exponential rate in time to the unique stationary profile. Our main result is that if the strength of the additional forces is sufficiently small, this global exponential equilibration behavior persists, at a slightly smaller rate. Our proof uses the formulation of the equation as a Wasserstein gradient flow, and an auxiliary lower-order Lyapunov functional. 
\end{abstract}

\maketitle

\section{
    Introduction}
We prove existence and global exponential convergence to equilibrium of non-negative unit-mass weak solutions to fourth-order evolution equations in $\R$ of the following form:
\begin{equation} \label{eq:thinfilm}
\partial_{t} \rho = -\left(\rho \,\rho_{xxx}\right)_x + \lambda \left(\rho\, x\right)_{x} + \eps \left(\rho\, h'(\rho )_{x}\right)_{x}.
\end{equation}
Throughout this paper, $\rho$ denotes a (time-dependent) probability density with bounded second moment. The confinement strength $\lambda$ is strictly positive. The non-linear function $h:\Rnn \to \R$ degenerates at $0$ up to second order. Moreover, it fulfills some additional assumptions on its higher derivatives specified in Hypothesis \ref{hyp:hderiv} below, but its precise form is kept quite general. For our exponential equilibration result we require the parameter $\eps > 0$ to be sufficiently small.

Thin-film equations similar to \eqref{eq:thinfilm} are known to appear as approximations to the motion of a spreading droplet of liquid on a solid surface. The fourth order term models the force generated by surface tension, see e.g. \cite{GO2_gf, GO1_approx, O_lub} for details. The lower-order terms describe additional forces acting on the system, such as gravity. Some examples for such models including stabilizing or destabilizing second-order terms include \cite{GPS2, GPS, L}.

Basic existence theory for thin-film type equations was developed in \cite{BF}, and has been extensively studied since then, see e.g. \cite{BDPG, GKO, GGKO, Gn, GP, G, LST}. Existence results for equations including second-order diffusion terms can be found in e.g. \cite{BP2_pormedia, BP3_ex_blowup, BP,  CEFFJ}. In the case of a bounded domain, qualitative results on long-time behavior and steady states have been obtained in e.g. \cite{BGW, BP2_pormedia, LP, LP2}. The main techniques in this case are based on the fact that in a bounded domain, solutions are strictly positive after a finite time, which is not the case in the unbounded domain studied in this paper.

We rigorously prove existence and quantitative equilibration estimates for weak solutions to equation \eqref{eq:thinfilm}, using its gradient-flow formulation, see e.g. \cite{CEFFJ, DM, GO2_gf, MMCS, O_lub}: On the metric space $\mpt(\R)$ of probability measures on $\R$ with finite second moment, with the Wasserstein-2 metric as distance, the underlying energy functional is given by 
\begin{equation}
    \me_\eps(\rho) = \left\{ 
    \begin{array}{ll}
        \int_\R \left[\frac{1}{2} \rho_x^2 + \frac{\lambda}{2}x^2\,\rho + \eps h(\rho)\right]\,\intd x & \text{for } \rho \in H^1(\R)\\
        +\infty & \text{else.}
    \end{array}
    \right..
\end{equation}
We first prove existence of a global-in-time weak solution to \eqref{eq:thinfilm}, given any initial datum $\rho_0$ with finite energy. Our main result, specified in Theorem \ref{thm:main_cvgce}, is to show that if the parameter $\eps > 0$ is small enough, this weak solution converges exponentially fast to a unique steady state. The rate of convergence only depends on $h, \lambda$ and $\eps$, but not on $\rho_0$.

For second-order evolution equations, such as the porous-medium equation \cite{JKO, O_pormed}, exponential convergence to equilibrium in often obtained by abstract gradient-flow theory \cite{AGS}, and relies on the geodesic convexity of the underlying energy functional. For thin-film type equations like \eqref{eq:thinfilm}, this machinery is not directly applicable, due to the lack of geodesic (semi-)convexity of the functional $\rho \mapsto \int \rho_x^2\,\intd x$, as observed in \cite{CS}. However, in the unperturbed case $\eps = 0$, exponential equilibration results have still been established previously \cite{CU, CT, MMCS}. The main idea these works are based on is a hidden connection between \eqref{eq:thinfilm} for $\eps = 0$ and a specific second-order porous medium equation, which is the Wasserstein gradient flow of the entropy functional
\begin{align*}
    \ml(\rho) := \int_\R \left[ \frac{2}{3} \rho^{3/2} + \frac{\tilde \lambda}{2}x^2 \,\rho \right]\,\intd x
\end{align*}
with the constant $\tilde \lambda = \sqrt{\lambda/6}$. While the energy $\me_0$ is non-convex, the functional $\ml$ is geodesically $\tilde \lambda$-convex, and its global minimizer is the Smyth-Hill profile
\begin{align} \label{eq:min_unperturbed}
    \bar \rho(x) = \frac{\lambda}{24} \left(x_*^2 - x^2\right)_+^2
\end{align}
with $x_* > 0$ chosen such that $\bar\rho$ has unit mass. The key observation now is that $\bar\rho$ is also the unique global minimizer of $\me_0$, making $\ml$ a natural functional to use in the asymptotic analysis of \eqref{eq:thinfilm} for $\eps = 0$. In \cite{MMCS}, a deeper connection between $\me_0$ and $\ml$ is used in to obtain equilibration estimates for fourth-order equations by means of abstract gradient flow theory. This method however relies on the precise structure of the evolution equation, which is not available in our setting with a general non-linear function $h$.

\subsection{Splitting the equation} Instead of using abstract theory, in \cite{CT} the connection between \eqref{eq:thinfilm} in the unperturbed case and the corresponding porous medium equation is established by using the following trick: The higher-order and lower-order terms from the equation are split in a particular way such that the decay of the entropy $\ml$ along the solution can be computed more easily. Specifically, \eqref{eq:thinfilm} with $\eps = 0$ is rewritten in the following form:
\begin{equation} \label{eq:tf_carrillo}
    \partial_{t} \rho = -2 \left(\rho^{3/2} \left(\sqrt{\rho} + \frac{\tilde \lambda}{2}x^2\right)_{xx}\right)_{xx} + 6\tilde \lambda\,\left(\rho \left( \sqrt{\rho} + \frac{\tilde \lambda}{2}x^2 \right)_x \right)_x,
\end{equation}
with $\tilde\lambda$ as above. The precise algebraic structure of \eqref{eq:tf_carrillo} now allows, at least formally, the direct derivation of a decay estimate for $\ml$ by integrating by parts:
\begin{align} \label{eq:ml_decay_intro}
    -\frac{\intd}{\intd t} \ml(\rho) = -\int_\R \partial_t \rho\,\left(\sqrt{\rho} + \frac{\tilde\lambda}{2}x^2 \right)\,\intd x \geq 6 \tilde \lambda \int_\R \rho \left( \sqrt{\rho} + \frac{\tilde \lambda}{2}x^2 \right)_x^2\,\intd x.
\end{align}
The main trick is that the term in the entropy production rate \eqref{eq:ml_decay_intro} corresponding to the fourth-order term in \eqref{eq:tf_carrillo} is clearly non-negative.
Using a standard entropy dissipation inequality for the $\tilde\lambda$-convex functional $\ml$, such as \cite[Theorem 2.1]{CMV}, estimate \eqref{eq:ml_decay_intro}
yields exponential decay of the relative entropy $\ml(\rho)- \ml(\bar\rho)$ at rate $12\tilde \lambda^2 = 2\lambda$.

Our approach to derive an exponential convergence rate for \eqref{eq:thinfilm} in the perturbed case $\eps > 0$ is a generalization of the formulation \eqref{eq:tf_carrillo}. The main challenge with this idea lies in the fact that the profile $\bar\rho$ is in general no longer a minimizer of the energy $\me_\eps$ over $\mpt(\R)$, and thus no longer a steady state for equation \eqref{eq:thinfilm}. In particular, $\ml$ can no longer be a Lyapunov functional for the evolution. To solve this problem, we need to perturb the auxiliary functional in a precise way: For sufficiently
small $\eps > 0$, we construct a family of uniformly convex outer potentials $W_\eps:\R \to \R$ with $W_0(x) = \frac{\tilde\lambda}{2}x^2$. Setting
\begin{equation} \label{eq:ml_intro}
    \ml_\eps(\rho) := \int_\R \left[\frac{2}{3}\rho^{3/2} + W_\eps \,\rho\right]\,\intd x,
\end{equation}
we show that for an appropriate choice of potential $W_\eps$, $\ml_\eps$ is a Lyapunov functional for \eqref{eq:thinfilm}. In particular, $\ml_\eps$ has the same unique global minimizer as $\me_\eps$.

The condition that the respective minimizers are identical poses a strong restriction on the choice of $W_\eps$: After having constructed a minimizer $\bar\rho^\eps$ of the energy $\me_\eps$, a variational argument shows that this condition requires us to choose, up to an additive constant, $W_\eps = -\sqrt{\bar\rho^\eps}$ on the set where $\bar\rho^\eps$ is positive. Observe that in the unperturbed case, this condition is satisfied by $W_0(x) = \frac{\tilde\lambda}{2}x^2$, which follows from \eqref{eq:min_unperturbed}.

In addition, we need $W_\eps$ to be uniformly convex on $\R$, in order to make the entropy $\ml_\eps$ uniformly geodesically convex. By the considerations above, we thus need to show that the profile $-\sqrt{\bar\rho^\eps}$ on the set $\{\bar\rho^\eps > 0\}$ can be extended to a potential $W_\eps$ that is uniformly convex with some positive modulus $\tilde\lambda_\eps$ on the whole real line, and sufficiently regular on the boundary points $\partial\{\bar\rho^\eps > 0\}$.

The main step of the convergence proof is the generalization of the splitting \eqref{eq:tf_carrillo}. For now, we present the formal ideas, the details and rigorous justification will be done in section \ref{sec:exp_cvgce}. We show that \eqref{eq:thinfilm} can be written in the following form:
\begin{equation} \label{eq:carrillo_general_intro}
     \partial_t\rho = -2\,\left(\rho^{3/2}\left(\sqrt{\rho} + W_\eps\right)_{xx}\right)_{xx} + 6\, \left(\rho\, (W_\eps)_{xx}\left(\sqrt{\rho} + W_\eps\right)_x\right)_x + \eps\,(\rho \,\mr_\eps)_x
\end{equation}
where $\mr_\eps$ is a remainder term that can be estimated well. More precisely, we prove a functional inequality of the form
\begin{align} \label{eq:funcineq_intro}
    \int_\R \rho\,\mr_\eps^2\,\intd x \leq \kappa \int_\R \rho\,\left(\sqrt{\rho} + W_\eps \right)_x^2\,\intd x
\end{align}
with a constant $\kappa$ that is independent of $\rho$ and (small) $\eps$.
Note that since $W_0 = \frac{\tilde\lambda}{2}(\cdot)^2$, the splitting \eqref{eq:carrillo_general_intro} is precisely the same as \eqref{eq:tf_carrillo} in the case $\eps = 0$. Similarly as in \eqref{eq:ml_decay_intro}, this yields an estimate for the decay of $\ml_\eps$, where the fourth-order terms can be omitted due to their non-negativity:
\begin{align*}
    -\frac{\intd}{\intd t} \ml_\eps(\rho)  = -\int_\R \partial_t \rho\,\left(\sqrt{\rho} + W_\eps \right)\,\intd x \geq 6 \tilde\lambda_\eps \int_\R \rho\,\left(\sqrt{\rho} + W_\eps \right)_x^2\,\intd x + \eps \int_\R \rho\,\mr_\eps\,\left(\sqrt{\rho} + W_\eps \right)_x\,\intd x
\end{align*}
The first integral above yields the desired entropy production rate in the same way as in the unperturbed case. Inequality \eqref{eq:funcineq_intro} and Young's inequality show that the second integral can be controlled by means of the first one. This implies a decay estimate for $\ml_\eps$ similar to \eqref{eq:ml_decay_intro},
proving exponential decay of $\ml_\eps(\rho)$ along the solution.

\subsection{Assumptions and main results}
Below, we specify our assumptions on the non-linear function $h: \Rnn \to \R$ in the lower-order term, and summarize our main results.
\begin{hypothesis} \label{hyp:hderiv}
We assume that $h \in C^2(\Rnn) \,\cap\, C^4(\Rp)$ with $h(0) = h'(0) = 0$. Further, we assume that for some constant $A < +\infty$, it holds
\begin{align} \label{eq:hderbound_second}
    |h''(r)| &\leq \frac{A}{2}\,\min\left\{\frac{1}{\sqrt{r}}, \sqrt{r}\right\}\quad \text{for all } r > 0.
\end{align}
Additionally, we assume that for any bound $H < +\infty$, there exists a constant $L_H$ such that
\begin{align} \label{eq:hderbound_fourth}
    |h^{(4)}(r)| &\leq \frac{L_H}{8\, r^{3/2}}  \quad \text{for } 0 < r \leq H.
\end{align}
\end{hypothesis}
\begin{remark}
    Note that we do not pose any assumptions on the sign of $h$, meaning that the lower-order term in \eqref{eq:thinfilm} may have stabilizing or destabilizing effects. The assumption that $h'(0) = 0$ is made to simplify the notation, but removing it does not change any of the results. This can be seen by observing that replacing $h$ by the function $\hat h(r) = h(r) + \alpha r$ with an arbitrary $\alpha \in \R$ only changes the energy functional $\me_\eps$ by a global constant. The assumptions \eqref{eq:hderbound_second} and \eqref{eq:hderbound_fourth} impose bounds on the growth of $h$ for small and large $r$. Specifically, if there exist $p,q \in \R$ such that $|h(r)| \sim r^p$ for $r \to 0$ and $|h(r)| \sim r^q$ for $r \to \infty$, \eqref{eq:hderbound_second} and \eqref{eq:hderbound_fourth} imply $p \geq 5/2$ and $q \leq 3/2$. As results such as \cite{BP, HR} suggest, fourth-order equations like \eqref{eq:thinfilm} may admit finite-time blow-up phenomena if $|h|$ grows too fast for $r \to \infty$, meaning that global existence can not be expected. However, we do not claim that the above conditions are sharp.
\end{remark}
The main results of this paper are summarized the following three theorems. We assume in all of the following that $\lambda > 0$ and a function $h: \Rnn \to \R$ that satisfies Hypothesis \ref{hyp:hderiv} are given.
\begin{theorem}[Existence of solution] \label{thm:main_ex}
    For all $\eps > 0$ and initial data $\rho_0 \in \mpt(\R) \cap H^1(\R)$, there exists a weakly continuous curve $\rho: [0, +\infty[ \,\to \mpt(\R)$ with $\rho(0) = \rho_0$ that satisfies equation \eqref{eq:thinfilm} in the following weak sense: We have $\rho \in L^2_{loc}([0, +\infty[, H^2(\R))$ and for every test function $\zeta \in C^\infty_c(]0, +\infty[\,\times \,\R)$, it holds
        \begin{align} \label{eq:weak_sol_intro}
            \int_0^\infty \int_\R \left[\rho\,\partial_t \zeta + (\rho\,\zeta_x)_x \left(-\rho_{xx} + \frac{\lambda}{2}x^2 + \eps h'(\rho)\right)\right]\,\intd x\,\intd t = 0.
        \end{align}
\end{theorem}
\begin{theorem}[Global minimizer] \label{thm:main_stat}
 The energy functional $\me_\eps$ possesses a global minimizer $\bar\rho^\eps$, which is radially decreasing about $0$, has bounded support and lies in $C^1(\R)$.
\end{theorem}
\begin{theorem}[Exponential convergence to equilibrium] \label{thm:main_cvgce}
    There exists some $\bar\eps > 0$ and a constant $C < +\infty$ such that for every $0 < \eps < \bar\eps$, the following hold:
    \begin{itemize}
        \item The minimizer $\bar\rho^\eps$ of the energy is unique.
        \item For all initial data $\rho_0 \in \mpt(\R) \cap H^1(\R)$, the weak solution $\rho$ to \eqref{eq:thinfilm} with $\rho(0) = \rho_0$ from Theorem \ref{thm:main_ex} converges to $\bar\rho^\eps$ in $L^1(\R)$ exponentially at rate $2\lambda - C\eps$. More precisely, there is a constant $c(\rho_0)$ depending on $\rho_0$ but not on $t$ such that for all $t \geq 0$:
        \begin{align*}
            \|\rho(t) - \bar\rho^\eps\|_{L^1(\R)}^2 \leq c(\rho_0)\, e^{-(2\lambda - C\eps)t}.
        \end{align*}
    \end{itemize}
\end{theorem}

\subsection{Outline} In section \ref{sec:sol_ex}, we prove Theorem \ref{thm:main_ex}. We construct a weak solution to \eqref{eq:thinfilm} by analyzing the minimizing movement scheme for the energy functional $\me_\eps$. The proof mainly follows the same steps as \cite{CEFFJ}, where an existence result analogous to Theorem \ref{thm:main_ex} has been proven for an equation similar to \eqref{eq:thinfilm}, with the function $h(r)$ being some power of $r$. We begin by proving some basic properties of the energy, after which we study the time-discrete minimizing movement steps. We conclude by proving convergence of the piecewise constant interpolations in some topology to a weakly continuous curve, and showing that this limit curve is a weak solution to \eqref{eq:thinfilm} in the sense that it satisfies \eqref{eq:weak_sol_intro}. 

In section \ref{sec:ex_min}, we prove Theorem \ref{thm:main_stat}, derive an Euler-Lagrange equation for the global minimizer $\bar\rho^\eps$, and prove some bounds on the derivatives of its square-root needed in the following steps.

Section \ref{sec:exp_cvgce} contains the proof of Theorem \ref{thm:main_cvgce}. We construct the modified entropy functional $\ml_\eps$ as described in the introduction and analyze its decay along the weak solution from section \ref{sec:sol_ex}, assuming that $\eps$ is sufficiently small. After proving the relevant functional inequalities, we arrive at the exponential convergence result from Theorem \ref{thm:main_cvgce}.

\subsection{Notations}For probability measures $\rho_0, \rho_1 \in \mpt(\R)$, we denote by $\wass(\rho_0, \rho_1)$ the Wasserstein-2 distance between $\rho_0$ and $\rho_1$. The second moment $\int x^2\rho\,\intd x$ of a probability measure $\rho$ will be denoted as $\mom_2[\rho]$. We further define the following notion of convergence: For a sequence $(\rho_n)_n \subset \mpt(\R) \cap H^1(\R)$, we say that $\rho_n$ converges weakly in $\mpt(\R) \cap H^1(\R)$ to some $\rho$ and write $\rho_n \weakto \rho$ whenever $\rho_n$ converges to $\rho$ narrowly as probability measures as well as weakly in $H^1(\R)$.

\section{Construction of weak solution as Wasserstein gradient flow}\label{sec:sol_ex}

\subsection{Basic properties of the energy functional}

We start by proving upper and lower bounds for the energy functional $\me_\eps$, as well as its weak lower semi-continuity.
\begin{proposition}\label{prop:energy_props}
For every $\eps > 0$, the energy functional $\me_\eps$ is bounded from below in $\mpt(\R)$. Furthermore, for every family $\ms \subset \mpt(\R)$ of probability measures, $\me_\eps(\rho)$ is bounded from above for $\rho \in \ms$ if and only if $\ms$ is a bounded subset of $H^1(\R)$ and the $\rho \in \ms$ have uniformly bounded second moments.
\end{proposition}
\begin{proof}
    By integrating \eqref{eq:hderbound_second} twice and using $h(0) = h'(0) = 0$, we have for all $r \geq 0$:
    \begin{align} \label{eq:hestim_firstorder}
        |h'(r)| \leq \frac{A}{2} \int_0^r \sqrt{s}\,\intd s = \frac{A}{3}r^{3/2}, \qquad |h(r)| \leq \int_0^r |h'(s)|\,\intd s \leq \frac{2A}{15}\,r^{5/2}.
    \end{align}
    This implies the following bound on the $h$-term in the energy:
    \begin{align*}
        \left|\int_\R h(\rho)\,\intd x \right| \leq \frac{2A}{15}\int_\R \rho^{5/2}\,\intd x = \frac{2A}{15} \|\rho\|_{L^{5/2}(\R)}^{5/2}.
    \end{align*}
    It follows from the Gagliardo-Nirenberg inequality that $\|\rho\|_{L^{5/2}(\R)}^{5/2} \leq C\,\|\rho_x\|_{L^2(\R)}$ for all $\rho \in \mpt(\R)$, with some constant $C < +\infty$. Inserting this into the estimate above, we obtain for all $\rho \in \mpt(\R) \cap H^1(\R)$:
    \begin{align} \label{eq:energy_basic_estim}
        \left|\me_\eps(\rho) - \frac{1}{2}\|\rho_x\|_{L^2(\R)}^2 - \frac{\lambda}{2} \mom_2[\rho]\right| = \eps\, \left|\int_\R h(\rho)\,\intd x \right| \leq \eps \hat C\, \|\rho_x\|_{L^2(\R)},
    \end{align}
    with the constant $\hat C = 2AC/15$. In particular, it holds
    \begin{align*}
        &\me_\eps(\rho) \geq \frac{1}{2}\|\rho_x\|_{L^2(\R)}^2 + \frac{\lambda}{2} \mom_2[\rho] - \eps \hat C\, \|\rho_x\|_{L^2(\R)} \\ 
        &= \frac{1}{4}\|\rho_x\|_{L^2(\R)}^2 + \frac{\lambda}{2} \mom_2[\rho] + \|\rho_x\|_{L^2(\R)} \left(\frac{1}{4}\|\rho_x\|_{L^2(\R)} - \eps \hat C \right) \geq \frac{1}{4}\|\rho_x\|_{L^2(\R)}^2 + \frac{\lambda}{2} \mom_2[\rho] - \eps^2 \hat C^2,
    \end{align*}
    where the last step follows from the elementary inequality $y\,\left(y/4 - b\right) \geq -b^2$ for $y,b \in \R$. This proves that $\me_\eps$ is bounded from below, and that for a family $\ms \subset \mpt(\R)$ such that $\me_\eps(\rho)$ is uniformly bounded from above for all $\rho \in \ms$, the expressions $\|\rho_x\|_{L^2(\R)}$ and $\mom_2[\rho]$ are uniformly bounded for all $\rho \in \ms$ as well. On the other hand, if $\|\rho_x\|_{L^2(\R)}$ and $\mom_2[\rho]$ are uniformly bounded for $\rho \in \ms$, then \eqref{eq:energy_basic_estim} implies
    \begin{align*}
        \me_\eps(\rho) \leq \frac{1}{2}\|\rho_x\|_{L^2(\R)}^2 + \frac{\lambda}{2} \mom_2[\rho] + \eps \hat C\, \|\rho_x\|_{L^2(\R)},
    \end{align*}
    which is bounded from above for $\rho \in \ms$. The claim now follows from the fact that boundedness of a family $\ms$ of probability measures in $H^1(\R)$ is equivalent to boundedness of $\|\rho_x\|_{L^2(\R)}$ for $\rho \in \ms$, which again follows from a Gagliardo-Nirenberg inequality.
\end{proof}
\begin{lemma} \label{lem:eeps_lsc}
    The functional $\me_\eps$ is lower semi-continuous with respect to weak convergence in $\mpt(\R) \cap H^1(\R)$.
\end{lemma}
\begin{proof}
    Take any sequence $(\rho_n)_n$ such that $\rho_n \weakto \rho$ weakly in $\mpt(\R) \cap H^1(\R)$, as defined in the introduction. By lower semi-continuity of the $L^2$-norm with respect to weak $L^2$-convergence, and lower semi-continuity of $\mom_2[\rho]$ with respect to narrow convergence, it directly follows
    \begin{align*}
        \int_\R \left[\frac{1}{2} \rho_x^2 + \frac{\lambda}{2}x^2 \rho\right]\,\intd x \leq \liminf_{n \to \infty} \int_\R \left[\frac{1}{2} (\rho_n)_x^2 + \frac{\lambda}{2}x^2 \rho_n\right]\,\intd x.
    \end{align*}
    In order to prove convergence of the $h$-terms, we first claim that $\rho_n \weakto \rho$ weakly in $\mpt(\R) \cap H^1(\R)$ implies $\rho_n \to \rho$ strongly in $L^2(\R)$. To prove the claim, we observe that the weak $H^1$-convergence $\rho_n \weakto \rho$ and Rellich's theorem imply $\rho_n \to \rho$ strongly in $L^2_{loc}(\R)$.  In order to derive the strong convergence in $L^2(\R)$, observe that for every $R < +\infty$, it holds
    \begin{align*}
        \limsup_{n \to \infty} \int_{\R \setminus [-R, R]} |\rho_n - \rho|^2\,\intd x \leq 2\, \left(\int_{\R \setminus [-R, R]} \rho^2\,\intd x + \limsup_{n \to \infty} \left(\|\rho_n\|_{L^\infty} \int_{\R \setminus [-R, R]} \rho_n\,\intd x \right) \right).
    \end{align*}
    The last expression can be made arbitrary small for $R \to +\infty$ by tightness of the sequence $(\rho_n)_n$, the fact that $\rho \in L^2(\R)$, and uniform boundedness of $\|\rho_n\|_{L^\infty(\R)}$, the latter following from the inequality $\|\rho\|_{L^\infty(\R)} \leq C \|\rho_x\|_{L^2(\R)}^{2/3}$ for probability measures $\rho$. Together with the strong convergence $\rho_n \to \rho$ in $L^2_{loc}(\R)$, this proves $\rho_n \to \rho$ strongly in $L^2(\R)$.

    We now obtain continuity of the $h$-term in the energy: With \eqref{eq:hestim_firstorder}, we have
    \begin{align*}
        &\int_\R |h(\rho_n) - h(\rho)|\,\intd x = \int_\R \left| \int_{\rho}^{\rho_n} h'(r)\,\intd r \right|\,\intd x \leq \frac{A}{3} \int_\R \left| \int_{\rho}^{\rho_n} r^{3/2}\,\intd r \right|\,\intd x \\
        &\leq \frac{A}{3} \max\left\{\|\rho_n\|_{L^\infty(\R)}^{1/2}, \|\rho\|_{L^\infty(\R)}^{1/2} \right\} \int_\R | \rho_n^2 - \rho^2|\,\intd x \leq C\, \|\rho_n - \rho\|_{L^2(\R)}\,\|\rho_n + \rho\|_{L^2(\R)},
    \end{align*}
    where $C < +\infty$ is a uniform upper bound on the prefactor in the expression before, which exists by the same argument as above. Since $\|\rho_n + \rho\|_{L^2(\R)}$ is bounded from above for all $n$ and $\|\rho_n - \rho\|_{L^2(\R)}$ tends to $0$, this shows $h(\rho_n) \to h(\rho)$ strongly in $L^1$, implying continuity of the $h$-term of the energy with respect to weak $\mpt(\R) \cap H^1(\R)$-convergence.
\end{proof}

\subsection{Analysis of the minimizing movement scheme}

We now study the minimizing movement scheme for the functional $\me_\eps$ in detail. For every positive time-step size $\tau > 0$ and every $\hat\rho \in \mpt(\R) \cap H^1(\R)$, we consider the functional $\eepstaurhat{\cdot}: \mpt(\R) \cap H^1(\R) \to \R$, defined as
\begin{align*}
    \eepstaurhat{\rho} := \me_\eps(\rho) + \frac{1}{2\tau} \wass^2(\rho, \hat\rho).
\end{align*}
We start by proving existence of a minimizer of $\eepstaurhat{\cdot}$ by the direct method of the Calculus of Variations, as well as a-priori regularity properties of this minimizer.
\begin{proposition}[Existence] \label{prop:step_min_ex}
    For every $\hat\rho \in \mpt(\R) \cap H^1(\R)$, the functional $\eepstaurhat{\cdot}$ possesses a global minimizer $\rho^+ \in \mpt(\R) \cap H^1(\R)$, and it holds
    \begin{align} \label{eq:step_energy_estimate}
        \frac{1}{2\tau}\wass^2(\rho^+, \hat\rho) \leq \me_\eps(\hat\rho) - \me_\eps(\rho^+).
    \end{align}
\end{proposition}
\begin{proof}
     It follows from Proposition \ref{prop:energy_props}, Alaoglu's theorem and Prokhorov's theorem that a minimizing sequence $(\rho_n)_n$ of the energy $\me_\eps$ is compact with respect to weak convergence in $\mpt(\R) \cap H^1(\R)$. The Wasserstein distance is lower-semicontinuous with respect to narrow convergence. In conjunction with Lemma \ref{lem:eeps_lsc}, the direct method of the Calculus of Variations yields existence of a global minimizer $\rho^+$. Estimate \eqref{eq:step_energy_estimate} follows immediately since $\eepstaurhat{\rho^+} \leq \eepstaurhat{\hat\rho} = \me_\eps(\hat\rho)$ by minimality.
\end{proof}
\begin{proposition}[Higher regularity] \label{prop:step_reg}
Let $\mh$ denote the entropy functional defined as $\mh(\rho) := \int_\R \rho\,\log\,\rho\,\intd x$. Then $\mh(\rho)$ is finite for all $\rho \in \mpt(\R) \cap H^1(\R)$, and for every $\hat\rho$ and $\rho^+$ as above, it holds
\begin{align} \label{eq:step_h2_bound}
    \int_\R (\rho^+_{xx})^2\,\intd x \leq \frac{\mh(\hat\rho) - \mh(\rho^+)}{\tau} + \lambda + \eps\, \frac{A}{2}\|\rho^+_x\|_{L^2(\R)}^2.
\end{align}
    In particular, the minimizer $\rho^+$ is in $H^2(\R)$.
\end{proposition}
\begin{proof}
    It follows from well-known estimates that $\rho \in H^1(\R)$ implies $\mh(\rho) < +\infty$ and $\mom_2[\rho] < +\infty$ implies $\mh(\rho) > -\infty$, proving that $\mh(\rho)$ is finite for every $\rho \in \mpt(\R) \cap H^1(\R)$.
    
    The proof of \eqref{eq:step_h2_bound} uses the flow interchange technique developed in \cite{MMCS}. We denote by $(\rho^s)_{s \geq 0}$ the solution to the heat flow $\partial_s \rho^s = \rho^s_{xx}$ with initial datum $\rho^0 = \rho^+$, and analyze the dissipation of the energy $\me_\eps$ along $(\rho^s)$. The flow interchange estimate \cite[Theorem 3.2]{MMCS} directly yields the estimate
    \begin{align*}
        \limsup_{s \to 0} \left[\frac{\me_\eps(\rho^+) - \me_\eps(\rho^s)}{s}\right] \leq \frac{\mh(\hat\rho) - \mh(\rho^+)}{\tau}.
    \end{align*}
    In order to estimate the derivative above, we write $\me_\eps(\rho) = \me_0(\rho) + \eps \int_\R h(\rho)\,\intd x$ and apply \cite[Lemma 4.4]{MMCS} to obtain an explicit estimate for the $\me_0$-term:
    \begin{align*}
        \limsup_{s \to 0} \left[\frac{\me_\eps(\rho^+) - \me_\eps(\rho^s)}{s}\right] &\geq \liminf_{s \to 0} \left[\frac{\me_0(\rho^+) - \me_0(\rho^s)}{s} \right] + \eps \limsup_{s \to 0} \int_\R \frac{h(\rho^+) - h(\rho^s)}{s}\,\intd x \\
        &\geq \int_\R (\rho^+_{xx})^2\,\intd x - \lambda + \eps \limsup_{s \to 0} \int_\R \frac{h(\rho^+) - h(\rho^s)}{s}\,\intd x.
    \end{align*}
    It remains to estimate the last integral. It follows from the properties of the heat flow that $\rho^s$ is smooth with $\rho^s_{xx} \in L^2(\R)$ for all $s > 0$. This implies with Fubini, integration by parts and \eqref{eq:hderbound_second} that for $0 < s_0 < s$, it holds
    \begin{align*}
        &\left|\int_\R \left[h(\rho^{s_0}) - h(\rho^s)\right]\,\intd x\right| = \left|\int_\R \int_{s_0}^s h'(\rho^\sigma)\,\rho^\sigma_{xx}\,\intd \sigma \intd x \right| = \left|\int_{s_0}^s \int_\R h'(\rho^\sigma)_x\,\rho^\sigma_x\,\intd x\intd \sigma \right| \\
        &= \left|\int_{s_0}^s \int_\R h''(\rho^\sigma)\,(\rho^\sigma_x)^2\,\intd x \right| \leq \frac{A}{2} \int_{s_0}^s \|\rho^\sigma_x\|_{L^2(\R)}^2\,\intd \sigma \leq \frac{A}{2} \|\rho^+_x\|_{L^2(\R)}^2\,(s - s_0),
    \end{align*}
    where the last estimate holds since the heat flow dissipates the Dirichlet energy. Taking the limit $s_0 \to 0$, it holds $\int_\R h(\rho^{s_0})\,\intd x \to \int_\R h(\rho^+)\,\intd x$, again by the properties of the heat flow. This implies that for every $s > 0$, it holds
    \begin{align*}
        \left|\int_\R \frac{h(\rho^+) - h(\rho^s)}{s}\,\intd x \right| \leq \frac{A}{2}\|\rho^+_x\|_{L^2(\R)}^2,
    \end{align*}
    proving \eqref{eq:step_h2_bound} by inserting into the estimates above.
\end{proof}
We next prove an Euler-Lagrange equation the minimizer $\rho^+$ satisfies. This result will be the starting point for proving that the limit curve constructed in subsection \ref{subsec:interpol} is a weak solution to \eqref{eq:thinfilm}.
\begin{proposition}[Euler-Lagrange equation] \label{prop:step_elg}
    Let $S \subset \R$ denote the open set $\{\rho^+ > 0\}$, and let $T: S \to \R$ denote the optimal transport map, pushing $\rho^+$ to $\hat\rho$. Then it holds
    \begin{align} \label{eq:step_elg}
        \frac{T - \mathrm{id}}{\tau} = \left(-\rho^+_{xx} + \frac{\lambda}{2} x^2 + \eps h'(\rho^+)\right)_x
    \end{align}
    almost everywhere in $S$, and $\rho^+ \in W^{3, \infty}_{loc}(S)$. If in addition, the set $\{\hat\rho > 0\}$ is bounded, then $S$ is bounded as well, and it holds $\rho^+ \in W^{3,\infty}(S)$.
\end{proposition}
\begin{proof}
    Since $\rho^+$ is continuous, $S$ is an open subset of $\R$, and thus it can be written as a countable union of disjoint open intervals:
\begin{align} \label{eq:s_interval_union}
    S = \biguplus_{n = 1}^\infty ]a_n, b_n[,
\end{align}
with $a_n, b_n \in \R \cup \{\pm \infty\}$. We fix any $n \in \N$ and choose an arbitrary $\eta \in C^\infty_c(]a_n, b_n[)$ with $\int \eta = 0$. There now exist points $a, b$ with $a_n < a < b < b_n$ such that $\eta = 0$ everywhere in $\R \setminus [a, b]$. By continuity, $\rho^+$ is bounded away from 0 on the interval $[a,b]$. Hence if $s_0 > 0$ is small enough, we have $\rho^s := \rho^+ + s \eta \geq 0$ in $\R$ for every $s \in [-s_0, s_0]$. In particular, $\rho^s \in \mpt(\R) \cap H^1(\R)$. Since $\rho^+$ is a global minimizer of the functional $\eepstaurhat{\cdot}$, this implies
\begin{align} \label{eq:eepstau_var_deriv}
    \left.\frac{\intd}{\intd s}\right|_{s = 0} \eepstaurhat{\rho^s} = \left.\frac{\intd}{\intd s}\right|_{s = 0} \left[\me_\eps(\rho^s) + \frac{1}{2 \tau} \wass^2(\rho^s, \hat\rho) \right] = 0,
\end{align}
if the derivative exists. We prove that this is indeed the case and derive an explicit expression for it. For the energy term, it holds for every $s \in [-s_0, s_0] \setminus \{0\}$:
\begin{align*}
    &\frac{\me_\eps(\rho^s) - \me_\eps(\rho^+)}{s} = \int_a^b \left[\frac{1}{2}\frac{\left(\rho^+_x + s \eta_x\right)^2 - (\rho^+_x)^2}{s} + \frac{\lambda}{2}x^2\eta + \eps \frac{h(\rho^+ + s\eta) - h(\rho^+)}{s} \right]\,\intd x \\
    &\to \int_a^b \left[ \eta_x\rho^+_x + \eta\,\left(\frac{\lambda}{2}x^2 + \eps h'(\rho^+)\right)  \right]\,\intd x = \int_a^b \eta\, \left(-\rho^+_{xx} + \frac{\lambda}{2}x^2 + \eps h'(\rho^+) \right) \,\intd x.
\end{align*}
The limit $s \to 0$ is justified by the dominated convergence theorem, since $h'$ is locally bounded. For the distance term, we apply Lemma \ref{lem:wass_deriv} proven in the appendix: Since $\hat\rho$ and $\rho^+$ are continuous in $\R$ and $\rho^+$ is positive in $]a_n, b_n[$, the map $s \mapsto \wass^2(\rho^s, \hat\rho)$ is differentiable at $s = 0$ with derivative
\begin{align*}
    \left.\frac{\intd}{\intd s}\right|_{s = 0} \wass^2(\rho^s, \hat\rho) = 2\int_{a_n}^{b_n} \eta\,\varphi\,\intd x,
\end{align*}
where $\varphi:\,]a_n, b_n[ \to \R$ is defined as
\begin{align*}
    \varphi(x) := \int_{x_0}^x \left(y - T(y)\right)\,\intd y,
\end{align*}
with an arbitrary $x_0 \in ]a_n, b_n[$, and $T \in L^\infty_{loc}(]a_n, b_n[)$ denoting the optimal transport map that pushes $\rho^+$ to $\hat\rho$.

We now insert these results into \eqref{eq:eepstau_var_deriv}, using that $\eta$ vanishes outside $]a,b[$, and obtain
\begin{align*}
    \int_{a_n}^{b_n} \eta\,\left(-\rho^+_{xx} + \frac{\lambda}{2}x^2 + \eps h'(\rho^+) + \frac{\varphi}{\tau}\right)\,\intd x = 0.
\end{align*}
Since $\eta \in C^\infty_c(]a_n, b_n[)$ was arbitrary, this yields that for some constant $C_n$, it holds
\begin{align} \label{eq:step_elg_integrated}
    -\rho^+_{xx} + \frac{\lambda}{2}x^2 + \eps h'(\rho^+) + \frac{\varphi}{\tau} \equiv C_n  \quad \text{ a. e. in } ]a_n, b_n[.
\end{align}
 Note that $\varphi$ is differentiable almost everywhere in $]a_n, b_n[$, with derivative $\mathrm{id} - T$. After subtracting $\frac{\varphi}{\tau}$ on both sides of \eqref{eq:step_elg_integrated}, the right-hand side is thus differentiable almost everywhere. Hence, the same holds for the left-hand side as well, and we obtain
\begin{align*}
    \left(-\rho^+_{xx} + \frac{\lambda}{2}x^2 + \eps h'(\rho^+) \right)_x = \frac{T - \mathrm{id}}{\tau} \quad \text{ a. e. in } ]a_n, b_n[.
\end{align*}
Since $n$ was arbitrary, this finishes the proof of \eqref{eq:step_elg}.

The claim $\rho^+ \in W^{3, \infty}_{loc}(S)$, follows directly from \eqref{eq:step_elg_integrated}, since $C_n - \frac{\lambda}{2}x^2 - \eps h'(\rho^+) - \frac{\varphi}{\tau}$ is in $W^{1,\infty}_{loc}(S)$ by construction of $\varphi$ and the fact that $T$ is locally bounded in $S$ by monotonicity.
    
We now multiply \eqref{eq:step_elg} by $\tau$ and rearrange the terms to obtain
\begin{align} \label{eq:rhopxxx_eq}
    \tau \rho^+_{xxx} = \left(\lambda \tau + 1\right) x - T(x) + \eps \tau h'(\rho^+)_x
\end{align}
almost everywhere in $S$. Assuming now that $\{\hat\rho > 0\}$ is bounded,
the last two terms are uniformly bounded in $S$: It holds $\hat\rho(T(x)) > 0$ everywhere in $S$, while Proposition \ref{prop:step_reg} implies that $h'(\rho^+)_x = h''(\rho^+) \rho^+_x$ is uniformly bounded in $\R$. Hence, the expression on the right-hand side of \eqref{eq:rhopxxx_eq} grows linearly as $|x| \to \infty$. Since $\rho^+$ is a probability density, this implies that $S$ is bounded, and thus $\rho^+_{xxx}$ is uniformly bounded in $S$, finishing the proof.
\end{proof}

This result is used in the following lemma, which shows that one step in the minimizing movement scheme satisfies a time-discrete version of the weak formulation \eqref{eq:weak_sol_intro} up to a small error.
\begin{lemma} \label{lem:step_pde_approx}
    Let $\zeta \in C^\infty_c(\R)$ be arbitrary. Then for every $\hat\rho$ and $\rho^+$ as before, it holds
    \begin{equation} \label{eq:step_pde_approx}
        \left| \int_\R \left[\zeta\, \frac{\rho^+ - \hat\rho}{\tau} - \left( \rho^+\zeta_x\right)_x\, \left(-\rho^+_{xx} + \frac{\lambda}{2}x^2 + \eps h'(\rho^+)\right) \right]\,\intd x \right| \leq \frac{1}{2\tau}\|\zeta_{xx}\|_{C^0(\R)}\,\wass^2(\rho^+, \hat\rho).
    \end{equation}
    \end{lemma}
    \begin{proof}
        Let $T$ denote the optimal transport map pushing $\rho^+$ to $\hat\rho$. In particular, it holds $\hat\rho = T\#\rho^+$. Thus, we obtain
        \begin{align*}
            &\left|\int_\R \left[ \zeta\, \frac{\rho^+-\hat\rho}{\tau} + \rho^+ \zeta_x\, \frac{T(x) - x}{\tau}\right]\,\intd x \right| = \left| \int_\R \left[ \rho^+ \,\frac{\zeta(x) - \zeta(T(x)) + \zeta_x(x) (T(x) - x)}{\tau}\right]\,\intd x\right| \\
            &\leq \frac{1}{2\tau} \|\zeta_{xx}\|_{C^0(\R)} \int_\R \rho^+ \,(T(x) - x)^2\,\intd x = \frac{1}{2\tau}\|\zeta_{xx}\|_{C^0(\R)}\,\wass^2(\rho^+, \hat\rho).
        \end{align*}
        The integrand $\rho^+\zeta_x \frac{T(x)-x}{\tau}$ makes sense, as $T$ is undefined only on the set $\{\rho^+ = 0\}$. Inserting expression \eqref{eq:step_elg} for $\frac{T(x)-x}{\tau}$, it holds
        \begin{align*}
            \int_\R \rho^+ \zeta_x\, \frac{T(x) - x}{\tau}\,\intd x &= \int_S \rho^+ \zeta_x\, \frac{T(x) - x}{\tau}\,\intd x = \int_S \rho^+\zeta_x\,\left(-\rho^+_{xx} + \frac{\lambda}{2} x^2 + \eps h'(\rho^+)\right)_x\,\intd x \\
            &= \sum_{n=1}^\infty \int_{a_n}^{b_n} \rho^+\zeta_x\,\left(-\rho^+_{xx} + \frac{\lambda}{2} x^2 + \eps h'(\rho^+)\right)_x\,\intd x
        \end{align*}
        In order to integrate by parts, we show that the boundary terms and $b_n$ vanish. The proof for $a_n$ is analogous. In the case that $b_n = +\infty$, this follows from the fact that $\rho^+$ has bounded second moment and that $\rho^+ \in H^2(\R)$ implies $\rho^+\rho^+_{xx} \in L^1(\R)$. For finite $b_n$, the only critical step is showing that $\rho^+\rho^+_{xx} \to 0$ at least along a subsequence $x \to b_n$. This follows again using $\rho^+ \in H^2(\R)$, since $\rho^+_{xx} \in L^2(\R)$ implies $\rho^+(x) \leq c\,(b_n - x)^{3/2}$ for some constant $c$ when $x$ approaches $b_n$ from below, as well as $\rho^+_{xx} \leq c\,(b_n - x)^{-1/2}$ along a subsequence $x \to b_n$.
        
        Thus, the boundary terms vanish, and we have
        \begin{align*}
            \int_\R \rho^+ \zeta_x\, \frac{T(x) - x}{\tau}\,\intd x = -\int_\R \left(\rho^+\zeta_x\right)_x\left(-\rho^+_{xx} + \frac{\lambda}{2}x^2 + \eps h'(\rho^+)\right)\,\intd x.
        \end{align*}
        Inserting this into the above estimate yields \eqref{eq:step_pde_approx}.
    \end{proof}

\subsection{Interpolation and convergence to weak solution} \label{subsec:interpol}

Given any initial datum $\rho_0$ with finite energy, we now construct for each $\tau > 0$ a piecewise constant curve $\rho^\tau: [0, \infty[ \to \mpt(\R) \cap H^1(\R)$ in the following typical way:
\begin{align*}
    \rho^\tau(t) := \hat\rho_{n+1} = \rho_n^+ \quad \text{for } n\tau < t \leq (n+1)\tau, \quad \rho^\tau(0) = \rho_0,
\end{align*}
where the sequences $(\hat\rho_n)_{n \geq 0}$ and $(\rho_n^+)_{n \geq 0}$ of densities $\hat\rho_n, \rho_n^+ \in \mpt(\R) \cap H^1(\R)$ are defined recursively as
\begin{align*}
    \hat\rho_0 := \rho_0, \quad \hat\rho_{n+1} := \rho_n^+ \in \argmin_{\rho \in \mpt(\R) \cap H^1(\R)} \eepstau(\rho|\hat\rho_n) \quad \text{for every } n \geq 0.
\end{align*}
We shall prove that along a subsequence $\tau \to 0$, these piecewise constant curves converge in some topology to a Hölder-continuous limit curve $\rho: [0, \infty[ \to \mpt(\R) \cap H^1(\R)$, and that this limit curve $\rho$ satisfies the weak formulation \eqref{eq:weak_sol_intro} of \eqref{eq:thinfilm}. The latter result will rely on the following lemma, which is essentially a time-integrated version of Lemma \ref{lem:step_pde_approx}.
\begin{lemma} \label{lem:discr_pde_approx}
    We fix an arbitrary $\zeta \in C^\infty_c(]0, \infty[ \times \R)$. For any $\tau > 0$, we define its time-discrete approximation $\zeta^\tau:[0, \infty[ \to C^\infty_c(\R)$ as $\zeta^\tau(t, x) := \zeta(n\tau, x)$ for every $t$ with $n\tau \leq t < (n+1)\tau$. Then the following estimate holds:
    \begin{align} \label{eq:discr_pde_approx}
        \left|\int_0^\infty \int_\R\left[\rho^\tau \partial_t \zeta + \left(\rho^\tau \zeta^\tau_x\right)_x \left(-\rho^\tau_{xx} + \frac{\lambda}{2}x^2 + \eps h'(\rho^\tau)\right)\right]\,\intd x\intd t\right| \leq \|\zeta\|_{C^2} \left(\me_\eps(\rho_0) - \inf_\rho\, \me_\eps(\rho)\right)\,\tau
    \end{align}
\end{lemma}
\begin{proof}
    We split the time-integral on the left-hand side into a sum of integrals over intervals of the form $]n\tau, (n+1)\tau[$. By construction of $\rho^\tau$ and Fubini, we obtain
    \begin{align*}
        &\int_0^\infty \int_\R \rho^\tau \partial_t\zeta\,\intd x\intd t = \sum_{n = 0}^\infty \int_{n\tau}^{(n+1)\tau} \int_\R \rho_n^+\partial_t\zeta\,\intd x\intd t = \sum_{n=0}^\infty \int_\R \rho_n^+\, \left(\zeta((n+1)\tau, x) - \zeta(n\tau, x)\right)\,\intd x \\
        &= -\int_\R \rho_0^+ \zeta(0, x)\,\intd x + \sum_{n=1}^\infty \int_\R \zeta(n\tau, x)\, (\rho_{n-1}^+ - \rho_n^+)\,\intd x = \tau \sum_{n=1}^\infty \int_\R \zeta(n\tau, x)\,\frac{\hat\rho_n - \rho_n^+}{\tau}\,\intd x,
    \end{align*}
    where in the last step, we used $\zeta(0, \cdot) = 0$ and $\rho^+_{n-1} = \hat\rho_n$. Similarly, the construction of $\zeta^\tau$ implies
    \begin{align*}
        &\int_0^\infty \int_\R \left(\rho^\tau \zeta^\tau_x\right)_x \left(-\rho^\tau_{xx} + \frac{\lambda}{2}x^2 + \eps h'(\rho^\tau)\right)\,\intd x\intd t \\ &= \tau \sum_{n = 1}^\infty \int_\R \left(\rho_n^+\, \zeta(\tau n, x)_x\right)_x \left(-(\rho_n^+)_{xx} + \frac{\lambda}{2}x^2 + \eps h'(\rho^+_n)\right)\,\intd x.
    \end{align*}
    Applying Lemma \ref{lem:step_pde_approx} to $\hat\rho_n, \rho_n^+$ and $\zeta_n := \zeta(n\tau, \cdot) \in C^\infty_c(\R)$ thus yields
    \begin{align*}
        &\left|\int_0^\infty \int_\R\left[\rho^\tau \partial_t \zeta + \left(\rho^\tau \zeta^\tau_x\right)_x \left(-\rho^\tau_{xx} + \frac{\lambda}{2}x^2 + \eps h'(\rho^\tau)\right)\right]\,\intd x\intd t\right| \\
        &\leq \tau \sum_{n=1}^\infty \left|\int_\R \left[ \zeta_n\, \frac{\hat\rho_n - \rho_n^+}{\tau}  + (\rho_n^+\, (\zeta_n)_x)_x \left(-(\rho_n^+)_{xx} + \frac{\lambda}{2}x^2 + \eps h'(\rho^+_n)\right) \right]\,\intd x\right| \\
        &\leq \tau \sum_{n=1}^\infty \frac{1}{2\tau} \|(\zeta_n)_{xx}\|_{C^0(\R)}\wass^2(\rho_n^+, \hat\rho_n) \leq \|\zeta\|_{C^2} \left(\me_\eps(\rho_0) - \inf_\rho\,\me_\eps(\rho)\right)\,\tau,
    \end{align*}
    where the last estimate follows from \eqref{eq:step_energy_estimate}.
\end{proof}
Our next goal is to prove compactness of the sequence $(\rho^\tau)_{\tau \downarrow 0}$ in a suitable topology, and then show that the limit curve of a converging subsequence solves \eqref{eq:thinfilm} weakly by passing to the limit $\tau \to 0$ in \eqref{eq:discr_pde_approx}. In order to obtain the required compactness, we need the technical estimates given by the next two lemmas. The first one provides a uniform $L^2$-bound for the second spatial derivatives $\rho^\tau_{xx}$, the second one controls the step sizes of $\rho^\tau$ in time.
\begin{lemma} \label{lem:rhotauxx_l2bound}
    For all $0 \leq T < +\infty$, there exists a constant $M_T < +\infty$, depending on $\rho_0$ but not on $\tau$, such that for $\tau > 0$ sufficiently small, it holds
    \begin{align*}
        \int_0^T \|\rho_{xx}^\tau(t)\|_{L^2(\R)}^2\,\intd t \leq M_T.
    \end{align*}
\end{lemma}
\begin{proof}
    We fix $T > 0$ and $\tau > 0$. Let $N \in \N_0$ be such that $(N-1)\tau < T \leq N\tau$. We use the fact that $\rho^\tau$ is piecewise constant in time and \eqref{eq:step_h2_bound} to obtain
    \begin{align*}
        &\int_0^T \|\rho_{xx}^\tau(t)\|_{L^2(\R)}^2\,\intd t \leq \int_0^{N\tau} \|\rho_{xx}^\tau(t)\|_{L^2(\R)}^2\,\intd t = \tau \sum_{n = 0}^{N-1} \int_\R (\rho_n^+)_{xx}^2\,\intd x \\
        &\leq \tau \sum_{n=0}^{N-1}\left[\frac{\mh(\hat\rho_n) - \mh(\rho_n^+)}{\tau} + \lambda + \eps \frac{A}{2}\|(\rho_n^+)_x\|_{L^2(\R)}^2\right] \\
        &\leq \mh(\rho_0) - \mh(\rho_{N-1}^+) + \lambda N \tau + \eps\tau \frac{A}{2}\sum_{n=0}^{N-1}\|(\rho_n^+)_x\|_{L^2(\R)}^2.
    \end{align*}
    Since $\me_\eps(\rho_{N - 1}^+) \leq \me_\eps(\rho_0)$, there exists a $\tau$-uniform bound on $\mom_2[\rho_{N-1}^+]$ by Proposition \ref{prop:energy_props}, implying that $\mh(\rho_{N-1}^+)$ is $\tau$-uniformly bounded from below. Since $N\tau \leq T + \tau$ by construction, and each summand in the last term can be controlled by $\me_\eps(\rho_{N - 1}^+)$ by Proposition \ref{prop:energy_props}, this proves the claim.
\end{proof}
\begin{lemma} \label{lem:discr_holder_estim}
    Let $0 \leq s \leq t$ be arbitrary. Then for all $\tau > 0$, it holds
    \begin{align} \label{eq:discr_holder_estim}
        \wass^2(\rho^\tau(s), \rho^\tau(t)) \leq 2 \left(\me_\eps(\rho_0) - \me_\eps(\bar\rho^\eps)\right)\left(t - s + \tau \right).
    \end{align}
\end{lemma}
\begin{proof}
    Let $N_0, N_1$ be integers such that $N_0 \tau < s \leq (N_0 + 1)\tau$ and $N_1 \tau < t \leq (N_1 + 1)\tau$. The triangle inequality for $\wass$ and \eqref{eq:step_energy_estimate} yield
    \begin{align*}
        &\wass^2(\rho^\tau(s), \rho^\tau(t)) = \wass^2(\hat\rho_{N_0 + 1}, \rho_{N_1}^+) \leq \left(\sum_{n = N_0 + 1}^{N_1} \wass(\hat{\rho}_n, \rho_n^+)\right)^2 \\
        &\leq (N_1 - N_0) \sum_{n = N_0 + 1}^{N_1} \wass^2(\hat{\rho}_n, \rho_n^+) \leq 2\tau\,(N_1 - N_0) \sum_{n = N_0 + 1}^{N_1} \left[\me_\eps(\hat\rho_n) - \me_\eps(\rho_n^+)\right] \\
        &\leq 2\tau\,(N_1 - N_0)\, (\me_\eps(\rho_0) - \me_\eps(\bar\rho^\eps)).
    \end{align*}
    Inserting $N_1 \tau \leq t$ and $N_0 \tau \geq s - \tau$ proves \eqref{eq:discr_holder_estim}.
\end{proof}
With the help of these estimates, we now prove compactness of the sequence of piecewise constant curves $(\rho^\tau)_{\tau \downarrow 0}$ in the topology specified in the following Lemma.
\begin{lemma} \label{lem:tau_cvgce}
    There exists a curve $\rho: [0, +\infty[ \to \mpt(\R) \cap H^1(\R)$ such that $\rho(t) \in H^2(\R)$ at almost every $t \geq 0$, and a subsequence $\tau \to 0$ along which the following hold:
    \begin{align*}
    \begin{array}{cl}
        \rho^\tau(t) \weakto \rho(t) & \text{weakly as probability measures for every } t > 0, \\
        \rho^\tau \to \rho & \text{strongly in } L^2([0, T] \times \R), \\
        h'(\rho^\tau) \to h'(\rho) & \text{strongly in } L^2([0, T] \times \R), \\
        \rho^\tau_x \to \rho_x & \text{strongly in } L^2([0, T] \times \R), \\
        \rho^\tau_{xx} \weakto \rho_{xx} & \text{weakly in } L^2([0, T] \times \R),
    \end{array}
    \end{align*}
    for all $0 < T < +\infty$. Moreover, the limit curve $\rho$ is Hölder-$\frac{1}{2}$-continuous with respect to the Wasserstein metric. Specifically, it holds for every $s, t \geq 0$
    \begin{align*}
        \wass^2(\rho(s), \rho(t)) \leq 2 \,\left(\me_\eps(\rho_0) - \me_\eps(\bar\rho^\eps)\right)\, |s - t|.
    \end{align*}
\end{lemma}
\begin{proof}
    For every $t \geq 0$, the expressions $\mom_2[\rho^\tau(t)]$ and $\|\rho^\tau(t)\|_{H^1(\R)}$ are $\tau$-uniformly bounded from above by Proposition \ref{prop:energy_props} and the decay of $\me_\eps(\rho^\tau)$. Since uniform boundedness of the second moments implies tightness, Prokhorov's theorem yields compactness of the family $(\rho^\tau(t))_{\tau \downarrow 0}$ with respect to weak convergence of probability measures for every $t \geq 0$. By a diagonal argument, this implies that for a subsequence $(\tau_n)_n \to 0$, the $\rho^{\tau_n}(t)$ converge weakly to some probability densities $\rho(t)$ at every $t \in [0, +\infty[\, \cap\, \Q$. We use Lemma \ref{lem:discr_holder_estim} to prove that $\rho$ can be extended to a Hölder-continuous curve on the whole interval $[0, +\infty[$, and that $\rho^{\tau_n}(t) \weakto \rho(t)$ at every $t \geq 0$: Lemma \ref{lem:discr_holder_estim} and the lower semi-continuity of the Wasserstein distance with respect to weak convergence yield for arbitrary $s, t \in [0, +\infty[\,\cap\,\Q$
    \begin{align*}
        \wass^2(\rho(s), \rho(t)) \leq \liminf_{n \to \infty} \wass^2(\rho^{\tau_n}(s), \rho^{\tau_n}(t)) \leq 2\,\left(\me_\eps(\rho_0) - \me_\eps(\bar\rho^\eps)\right)\,|t - s|.
    \end{align*}
    Hence, $\rho$ is Hölder-continuous on $[0, +\infty[\,\cap\,\Q$. Since $\mpt(\R)$ with the Wasserstein metric is a complete metric space, there exists a unique Hölder-continuous extension $\rho: [0, +\infty[ \to \mpt(\R)$. The fact that for this extension, it holds $\rho^{\tau_n}(t) \weakto \rho(t)$ for every real $t \geq 0$ follows from weak compactness of the sequence $(\rho^{\tau_n}(t))_n$ in $\mpt(\R)$: Let $(\tau_{n_m})_m$ be a subsequence along which $\rho^{\tau_{n_m}}(t)$ converges weakly to some $\mu \in \mpt(\R)$. Then it holds for all $s \in \Q$
    \begin{align*}
        &\wass^2(\mu, \rho(t)) \leq 2 \left(\wass^2(\mu, \rho(s)) + \wass^2(\rho(s), \rho(t))\right)\\
        &\leq 2 \left(\wass^2(\rho(s), \rho(t)) +  \liminf_{m \to \infty} \wass^2(\rho^{\tau_{n_m}}(t), \rho^{\tau_{n_m}}(s)) \right) \leq 8\,\left(\me_\eps(\rho_0) - \me_\eps(\bar\rho^\eps)\right)\,|t - s|.
    \end{align*}
    Since $s \in \Q$ was arbitrary, this implies $\mu = \rho(t)$, proving that $\rho^{\tau_n}(t) \weakto \rho(t)$ weakly at every $t \geq 0$.

    In order to prove $\rho^{\tau_n} \to \rho$ strongly in $L^2([0, T] \times \R)$, we use the $\tau$-uniform bound on $\|\rho^\tau(t)\|_{H^1(\R)}$ at every $t \geq 0$.
    Banach-Alaoglu directly yields compactness of the sequence $(\rho^{\tau_n}(t))_n$ at every $t \geq 0$ in the weak $H^1(\R)$-topology. Since $\rho^{\tau_n}(t) \weakto \rho(t)$ weakly as probability measures, this implies $\rho(t) \in H^1(\R)$ and $\rho^{\tau_n}(t) \weakto \rho(t)$ weakly in $H^1(\R)$, and thus by Rellich's theorem, $\rho^{\tau_n}(t) \to \rho(t)$ strongly in $L^2_{loc}(\R)$ at every $t \geq 0$. In order to derive from this the strong convergence in $L^2(\R)$, observe that for every $R < +\infty$, it holds
    \begin{align*}
        &\limsup_{n \to \infty} \int_{\R \setminus [-R, R]} |\rho^{\tau_n}(t) - \rho(t)|^2\,\intd x \\
        &\leq 2\, \left(\int_{\R \setminus [-R, R]} \rho(t)^2\,\intd x + \limsup_{n \to \infty} \left(\|\rho^{\tau_n}(t)\|_{L^\infty(\R)} \int_{\R \setminus [-R, R]} \rho^{\tau_n}(t)\,\intd x \right) \right).
    \end{align*}
    The last expression can be made arbitrary small for $R \to +\infty$ by tightness of $\rho^{\tau_n}(t)$, uniform boundedness of $\|\rho^{\tau_n}(t)\|_{L^\infty(\R)}$ and the fact that $\rho(t) \in L^2(\R)$. Thus we have proven $\rho^{\tau_n}(t) \to \rho(t)$ strongly in $L^2(\R)$ at every $t \geq 0$. The strong convergence $\rho^{\tau_n} \to \rho$ in $L^2([0, T] \times \R)$ now follows from the dominated convergence theorem, since $\|\rho^{\tau_n}(t)\|_{L^2(\R)}$ is bounded uniformly with respect to $t \in [0, T]$.

    The strong $L^2([0,T] \times \R)$-convergence of $h'(\rho^{\tau_n})$ to $h'(\rho)$ follows from Hypothesis \ref{hyp:hderiv} and the $L^2$-convergence $\rho^{\tau_n} \to \rho$ by observing that
    \begin{align*}
        &\int_0^T \int_\R (h'(\rho^{\tau_n}) - h'(\rho))^2\,\intd x\,\intd t \leq \int_0^T \int_\R \left(\sup_{r \in [\rho, \rho^{\tau_n}]}|h''(r)|\right)^2 (\rho^{\tau_n} - \rho)^2\,\intd x\,\intd t \\
        &\leq \frac{L^2}{4} \max \left\{\|\rho\|_{L^\infty([0, T] \times \R)}, \|\rho^{\tau_n}\|_{L^\infty([0, T] \times \R)}\right\}\, \|\rho^{\tau_n} - \rho\|_{L^2([0, T] \times \R)},
    \end{align*}
    again using uniform boundedness of $\|\rho^{\tau_n}\|_{L^\infty}$.

    We next prove that $\rho_{xx} \in L^2([0, T] \times \R)$ and that $\rho^{\tau_n}_{xx} \weakto \rho_{xx}$ weakly in $L^2$. Lemma \ref{lem:rhotauxx_l2bound} gives a $\tau$-uniform bound on $\|\rho^\tau_{xx}\|_{L^2([0, T] \times \R)}$, and thus boundedness of $(\rho^{\tau_n})_n$ in the Banach space $L^2([0, T], H^2(\R))$. Hence $\rho^{\tau_n}$ converges weakly to some $u \in L^2([0, T], H^2(\R))$ up to subsequence. The identification $u = \rho$ follows directly from the fact that $\rho^{\tau_n} \to \rho$ in $L^2$. Thus $\rho \in L^2([0, T], H^2(\R))$ and $\rho^{\tau_n} \weakto \rho$ weakly in $L^2([0, T], H^2(\R))$. In particular, this implies $\rho_{xx} \in L^2$ at almost every $t \in [0, T]$, and $\rho^{\tau_n}_{xx} \weakto \rho_{xx}$ weakly in $L^2$.

    The strong convergence $\rho^{\tau_n}_x \to \rho_x$ follows from the strong $L^2$-convergence $\rho^{\tau_n} \to \rho$ and the weak $L^2$-convergence $\rho^{\tau_n}_{xx} \weakto \rho_{xx}$. Indeed, it holds for every $n$:
    \begin{align*}
        &\int_0^T \int_\R (\rho^{\tau_n}_x - \rho_x)^2\,\intd x\intd t = \int_0^T \int_\R (\rho^{\tau_n} - \rho)_x\, (\rho^{\tau_n} - \rho)_x\,\intd x\intd t \\
        &= -\int_0^T \int_\R (\rho^{\tau_n} - \rho)_{xx}\,(\rho^{\tau_n} - \rho)\,\intd x\intd t \leq \|\rho^{\tau_n}_{xx} - \rho_{xx}\|_{L^2([0,T]\times\R)} \,\|\rho^{\tau_n} - \rho\|_{L^2([0,T]\times\R)}.
    \end{align*}
    The first norm is uniformly bounded for all $n$, the second one converges to $0$. This yields $\rho^{\tau_n}_x \to \rho_x$ strongly in $L^2$, finishing the proof.
\end{proof}
As a direct consequence of the convergences proven in Lemma \ref{lem:tau_cvgce}, we obtain the result that $\rho$ is a weak solution to \eqref{eq:thinfilm}, finishing the proof of Theorem \ref{thm:main_ex}.
\begin{corollary}
    The curve $\rho$ constructed in Lemma \ref{lem:tau_cvgce} satisfies \eqref{eq:weak_sol_intro}.
\end{corollary}
\begin{proof}
    The claim follows from Lemmas \ref{lem:discr_pde_approx} and \ref{lem:tau_cvgce} by passing to the limit $\tau \to 0$ in \eqref{eq:discr_pde_approx}. Note that since $\zeta \in C^\infty_c$, we have $\zeta^\tau \to \zeta$ uniformly in $[0, +\infty[ \times \R$. The convergence of the integral on the left-hand side follows from the convergences proven in Lemma \ref{lem:tau_cvgce}.
\end{proof}

\section{Global minimizer of the energy}\label{sec:ex_min}
\subsection{Existence}
We now show that the energy functional $\me_\eps$ has a global minimizer $\bar\rho^\eps$ by the direct method. We then continue by analyzing some properties of this minimizer, proving Theorem \ref{thm:main_stat}. These results will be used in section \ref{sec:exp_cvgce} to prove our main result, the exponential convergence of the solutions constructed in section \ref{sec:sol_ex} to $\bar\rho^\eps$.

\begin{proposition}[Existence] \label{prop:min_ex_tf}
    For every $\eps > 0$, the functional $\me_\eps$ possesses a global minimizer $\bar\rho^\eps \in \mpt(\R)$, which is radially decreasing about $0$.
\end{proposition}
\begin{proof}
    Existence of a global minimizer follows from the direct method by the same arguments as in Proposition \ref{prop:step_min_ex}. In order to show that $\bar\rho^\eps$ is radially decreasing about $0$, let $\rho^*$ denote the radially decreasing rearrangement of $\bar\rho^\eps$. It holds $\mom_2[\rho^*] \leq \mom_2[\bar\rho^\eps]$ with equality if and only if $\bar\rho^\eps$ is radially decreasing. By the P\'{o}lya-Szeg\H{o} inequality, it holds $\|\rho^*_x\|_{L^2(\R)} \leq \|\bar\rho^\eps_x\|_{L^2(\R)}$, and the $h$-term is invariant under rearrangement. Thus we have
    $\me_\eps(\rho^*) \leq \me_\eps(\bar\rho^\eps)$
    with equality if and only if $\bar\rho^\eps$ is radially decreasing. By minimality of $\bar\rho^\eps$, this proves the claim.
\end{proof}
Note that we do not claim uniqueness of the minimizer $\bar\rho^\eps$ at this stage. In all of the following, we consider $\bar\rho^\eps$ to be some fixed choice of global minimizer for every $\eps > 0$. Since $\bar\rho^\eps$ is radially decreasing about $0$, there exists some $0 < x_*^\eps \leq +\infty$ such that $\bar\rho^\eps > 0$ in $]-x_*^\eps, x_*^\eps[$ and $\bar\rho^\eps = 0$ everywhere else.
\begin{proposition}[First-order optimality conditions] \label{prop:min_elg_tf}
     For all $\eps > 0$, we have $x_*^\eps < +\infty$, thus $\bar\rho^\eps$ is compactly supported. It holds $\bar\rho^\eps \in C^1(\R) \cap C^5(]-x_*^\eps, x_*^\eps[)$. In particular, $\bar\rho^\eps$ vanishes up to first order at $\pm x_*^\eps$, i.e.,
    \begin{equation*}
        \bar\rho^\eps(-x_*^\eps) = \bar\rho^\eps(x_*^\eps) = \bar\rho^\eps_x(-x_*^\eps) = \bar\rho^\eps_x(x_*^\eps) = 0.
    \end{equation*}
    Moreover, there is a constant $C_\eps$ such that
    \begin{equation} \label{eq:elg_tf}
        -\bar\rho^\eps_{xx} +\frac{\lambda}{2}x^2 + \eps h'(\bar\rho^\eps) = C_\eps \quad \text{ in } ]-x_*^\eps, x_*^\eps[.
    \end{equation}
\end{proposition}
\begin{proof}
    We apply Propositions \ref{prop:step_reg} and \ref{prop:step_elg} to the case $\rho^+ = \hat\rho = \bar\rho^\eps$. The Euler-Lagrange equation \eqref{eq:elg_tf} follows directly from Proposition \ref{prop:step_elg}, since we have $T = \mathrm{id}$ and the set $]-x_*^\eps, x_*^\eps[$ where $\bar\rho^\eps$ is positive is connected.

    The claim that $x_*^\eps$ is finite follows from \eqref{eq:elg_tf} by contradiction: Since $h'(\bar\rho^\eps) - C_\eps \in L^\infty(\R)$, the expression $\frac{\lambda}{2} x^2 + \eps h'(\bar\rho^\eps) - C_\eps$ tends to $+\infty$ as $|x| \to +\infty$. This however can not be the case for $\bar\rho^\eps_{xx}$, since it is a probability density.
    
    The claimed $C^1$-regularity holds by Proposition \ref{prop:step_reg}, as $\bar\rho^\eps \in H^2(\R) \subset C^1(\R)$. By non-negativity of $\bar\rho^\eps$, this implies that $\bar\rho^\eps$ vanishes up to first order at the boundary points $\pm x_*^\eps$ of the interval where it is positive. Inside the interval $]-x_*^\eps, x_*^\eps[$, we use \eqref{eq:elg_tf} and the fact that $h' \in C^3(\Rp)$ by Hypothesis \ref{hyp:hderiv} to conclude that $\bar\rho^\eps$ is indeed in $C^5(]-x_*^\eps, x_*^\eps[)$.
\end{proof}
By combining these results, we show that the minimizer $\bar\rho^\eps$ can be expressed as a Smyth-Hill profile, perturbed by an expression of order $\eps$. This is consistent with the form \eqref{eq:min_unperturbed} in the unperturbed case $\eps = 0$ seen in the introduction.
\begin{lemma}\label{lem:min_explicit}
For every $\eps > 0$ and every $x \in [-x_*^\eps, x_*^\eps]$, it holds
\begin{equation} \label{eq:min_explicit}
\bar\rho^\eps(x) = \frac{\lambda}{24}\left(x^2-(x_*^\eps)^2\right)^2 + \eps \eta^\eps(x)
\end{equation}
where $\eta^\eps: [-x_*^\eps, x_*^\eps] \to \R$ satisfies
\begin{equation*}
\eta^\eps(-x_*^\eps) = \eta^\eps(x_*^\eps) = \eta^\eps_x(-x_*^\eps) = \eta^\eps_x(x_*^\eps) = 0, \quad \eta^\eps_{xx} = h'(\bar\rho^\eps) - \frac{1}{2x_*^\eps}\int_{-x_*^\eps}^{x_*^\eps} h'(\bar\rho^\eps)\,\intd x.
\end{equation*}
\end{lemma}
\begin{proof}
    We start by deriving an expression for the constant $C_\eps$ in the Euler-Lagrange equation \eqref{eq:elg_tf}: Using the fact that $\bar\rho^\eps_x(-x_*^\eps) = \bar\rho^\eps_x(x_*^\eps) = 0$ and integrating \eqref{eq:elg_tf} yields
    \begin{align*}
        0 &= \bar\rho^\eps_x(x_*^\eps) - \bar\rho^\eps_x(-x_*^\eps) = \int_{-x_*^\eps}^{x_*^\eps} \bar\rho^\eps_{xx} \,\intd x = \int_{-x_*^\eps}^{x_*^\eps} \left[\frac{\lambda}{2}x^2 + \eps h'(\bar\rho^\eps) - C_\eps\right]\,\intd x \\
        &= \frac{\lambda}{3}(x_*^\eps)^3 - 2 C_\eps x_*^\eps  + \eps \int_{-x_*^\eps}^{x_*^\eps} h'(\bar\rho^\eps)\,\intd x
    \end{align*}
    By solving this equation for $C_\eps$, we obtain
    \begin{equation} \label{eq:c_eps_explicit}
        C_\eps = \frac{\lambda}{6}(x_*^\eps)^2 + \frac{\eps}{2x_*^\eps} \int_{-x_*^\eps}^{x_*^\eps} h'(\bar\rho^\eps)\,\intd x.
    \end{equation}
    We now define $\eta^\eps: [-x_*^\eps, x_*^\eps] \to \R$ to be the solution to the initial value problem
    \begin{align*}
        \eta^\eps(-x_*^\eps) = \eta^\eps_x(-x_*^\eps) = 0, \quad \eta^\eps_{xx} = h'(\bar\rho^\eps) - \frac{1}{2x_*^\eps}\int_{-x_*^\eps}^{x_*^\eps} h'(\bar\rho^\eps)\,\intd x,
    \end{align*}
    We claim that $\eta^\eps$ satisfies $\eta^\eps(x_*^\eps) = \eta^\eps_x(x_*^\eps) = 0$. Indeed, integrating the right-hand side once yields
    \begin{align*}
        \eta^\eps_x(x) = \int_{-x_*^\eps}^{x} h'(\bar\rho^\eps)\,\intd x - \frac{x + x_*^\eps}{2x_*^\eps}\int_{-x_*^\eps}^{x_*^\eps} h'(\bar\rho^\eps)\,\intd x
    \end{align*}
    at every $x \in [-x_*^\eps, x_*^\eps]$. Inserting $x = x_*^\eps$ yields $\eta^\eps_x(x_*^\eps) = 0$. To show $\eta^\eps(x_*^\eps) = 0$, we use the fact that $\bar\rho^\eps$, and thus also $h'(\bar\rho^\eps)$, is symmetric about 0. This implies
    \begin{align*}
        \eta^\eps_x(0) = \int_{-x_*^\eps}^{0} h'(\bar\rho^\eps)\,\intd x - \frac{1}{2}\int_{-x_*^\eps}^{x_*^\eps} h'(\bar\rho^\eps)\,\intd x = 0.
    \end{align*}
    Since $\eta^\eps_{xx}$ is an even function, this implies that $\eta^\eps_x$ is an odd function, and hence that $\eta^\eps$ is even. Thus, it also holds $\eta^\eps(x_*^\eps) = \eta^\eps(-x_*^\eps) = 0$, proving the claimed properties of $\eta^\eps$.

    Now, consider the expression on the right-hand side of \eqref{eq:min_explicit}, denoted as
    \begin{align*}
        r(x) := \frac{\lambda}{24}\left(x^2-(x_*^\eps)^2\right)^2 + \eps \eta^\eps(x).
    \end{align*}
    Our claim is that $\bar\rho^\eps = r$ in $[-x_*^\eps, x_*^\eps]$. We prove the claim by comparison of the second derivatives: Differentiating $r(x)$ twice yields
    \begin{align*}
        r_x(x) &= \frac{\lambda}{6} x \left(x^2 - (x_*^\eps)^2\right) + \eps \eta^\eps_x(x) \\
        r_{xx}(x) &= \frac{\lambda}{2} x^2 - \frac{\lambda}{6} (x_*^\eps)^2 +\eps \eta^\eps_{xx}(x).
    \end{align*}
    Inserting the expression for $\eta^\eps_{xx}$ and using \eqref{eq:c_eps_explicit} and \eqref{eq:elg_tf}, we have
    \begin{align*}
         r_{xx}(x) = \frac{\lambda}{2} x^2 - \frac{\lambda}{6} (x_*^\eps)^2 +\eps \left(h'(\bar\rho^\eps) - \frac{1}{2x_*^\eps}\int_{-x_*^\eps}^{x_*^\eps} h'(\bar\rho^\eps)\,\intd x\right) = \frac{\lambda}{2} x^2 + \eps h'(\bar\rho^\eps) - C_\eps = \bar\rho^\eps_{xx}(x)
    \end{align*}
    for all $x \in [-x_*^\eps, x_*^\eps]$. Since it holds $r(x_*^\eps) = \bar\rho^\eps(x_*^\eps) = 0$ and $r_x(x_*^\eps) = \bar\rho^\eps_x(x_*^\eps) = 0$, this proves the claim.
\end{proof}
\begin{remark}
    As a useful corollary, we obtain the fact that 
    \begin{align} \label{eq:x_eps_bound}
        0 < \liminf_{\eps \to 0} x_*^\eps \leq \limsup_{\eps \to 0} x_*^\eps < +\infty.
    \end{align}
    To see this, observe that for small $\eps$, the $L^\infty$-norm of $\bar\rho^\eps$ stays uniformly bounded, hence by the formula given in Lemma \ref{lem:min_explicit} and the local boundedness of $h'$, there is an $\eps$-independent constant $C$ such that $|\eta_{xx}^\eps| \leq C$ for all small enough $\eps$. Since $\eta^\eps$ vanishes up to first order at $\pm x_*^\eps$, this implies an estimate of the form
    \begin{align*}
        |\eta^\eps(x)| \leq C \,((x_*^\eps)^2 - x^2)^2
    \end{align*}
    for all $x \in [-x_*^\eps, x_*^\eps]$. Together with \eqref{eq:min_explicit} and the fact that $\bar\rho^\eps$ has mass one, this yields \eqref{eq:x_eps_bound}.
\end{remark}
\subsection{Higher derivatives of the square-root}
As explained in the introduction, our construction of a Lyapunov-functional for equation \eqref{eq:thinfilm} relies on the existence of some uniformly convex potential $W_\eps: \R \to \R$ with $W_\eps = -\sqrt{\bar\rho^\eps}$ inside the set where $\bar\rho^\eps$ is positive. This idea motivates the following result, which proves that if $\eps$ is sufficiently small, the profile $-\sqrt{\bar\rho^\eps}$ is uniformly convex on its support, with its third derivatives being of order $\eps$.
\begin{lemma} \label{prop:min_xx_bounds}
    There is a constant $K$ such that for every sufficiently small $\eps > 0$, it holds
    \begin{equation} \label{eq:min_xx_bound}
        \tilde\lambda_\eps := \sqrt{\frac{\lambda}{6}} - K\eps \leq -(\sqrt{\bar\rho^\eps})_{xx} \leq \sqrt{\frac{\lambda}{6}} + K\eps
    \end{equation}
    everywhere in $]-x_*^\eps, x_*^\eps[$. Additionally, there is a constant $M$ such that 
    \begin{align} \label{eq:min_xxx_bound}
    |(\sqrt{\bar\rho^\eps})_{xxx}| \leq M \eps
    \end{align}
    in $]-x_*^\eps, x_*^\eps[$.
\end{lemma}
\begin{proof}
    We introduce $u^\eps:]-x_*^\eps, x_*^\eps[ \to \R_{>0}$, defined as
    \begin{align*}
        u^\eps(x) := \frac{\bar\rho^\eps(x)}{\left(x^2 - (x_*^\eps)^2\right)^2} = \frac{\lambda}{24} + \eps \frac{\eta^\eps(x)}{\left(x^2 - (x_*^\eps)^2\right)^2},
    \end{align*}
    where the equality follows from Lemma \ref{lem:min_explicit}.
    For the higher derivatives of $\sqrt{\bar\rho^\eps}$, which exist in $]-x_*^\eps, x_*^\eps[$ by Proposition \ref{prop:min_elg_tf}, we obtain with $\sqrt{\bar\rho^\eps} = \left((x_*^\eps)^2 - x^2\right)\,\sqrt{u^\eps}$:
    \begin{align*} &\left(\sqrt{\bar\rho^\eps}\right)_{xx} = -2 \sqrt{u^\eps} - 4x (\sqrt{u^\eps})_x + \left((x_*^\eps)^2 - x^2\right) (\sqrt{u^\eps})_{xx}, \\   &\left(\sqrt{\bar\rho^\eps}\right)_{xxx} = -6 (\sqrt{u^\eps})_x - 6x (\sqrt{u^\eps})_{xx} + \left((x_*^\eps)^2 - x^2\right)(\sqrt{u^\eps})_{xxx}.
    \end{align*}
    The claim of the Lemma thus follows once we prove that the estimates
    \begin{align} \label{eq:sqrtueps_estim}
        &\left|\sqrt{u^\eps} -  \frac{1}{2}\sqrt{\frac{\lambda}{6}} \right| \leq C\eps \\ \label{eq:sqrtueps_x_estim}
        &\max\left\{|(\sqrt{u^\eps})_x|, |(\sqrt{u^\eps})_{xx}|, |(\sqrt{u^\eps})_{xxx}|\right\} \leq C \eps
    \end{align}
    hold in $]-x_*^\eps, x_*^\eps[$ with some constant $C < +\infty$. Since \eqref{eq:sqrtueps_estim} implies that $\sqrt{u^\eps}$ is $\eps$-uniformly bounded from below for small $\eps$, \eqref{eq:sqrtueps_x_estim} follows from \eqref{eq:sqrtueps_estim} and an estimate of the form
    \begin{equation} \label{eq:ueps_x_estim}
        \max\left\{|u^\eps_x|, |u^\eps_{xx}|, |u^\eps_{xxx}|\right\} \leq C\eps.
    \end{equation}
    Hence, it is sufficient to prove \eqref{eq:sqrtueps_estim} and \eqref{eq:ueps_x_estim}. By construction of $u^\eps$, this is equivalent to
    \begin{equation} \label{eq:etaeps_deriv_estim}
        \left|\left(\frac{\eta^\eps(x)}{\left(x^2 - (x_*^\eps)^2\right)^2}\right)^{(k)}\right| \leq C \quad \text{ in } ]-x_*^\eps, x_*^\eps[ \text{ for } k = 0,1,2,3
    \end{equation}
    as \eqref{eq:sqrtueps_estim} is equivalent to \eqref{eq:etaeps_deriv_estim} for $k = 0$, and \eqref{eq:ueps_x_estim} corresponds to $k = 1,2,3$ in \eqref{eq:etaeps_deriv_estim}. We shall derive \eqref{eq:etaeps_deriv_estim} from bounds on higher derivatives of $h'(\bar\rho^\eps)$. We first show that the implication
    \begin{equation} \label{eq:deriv_implication}
    \begin{split}
        &|h'(\bar\rho^\eps)^{(k)}| \text{ uniformly bounded in } ]-x_*^\eps, x_*^\eps[ \\
        \Rightarrow \quad &\left|\left(\frac{\eta^\eps(x)}{\left(x^2 - (x_*^\eps)^2\right)^2}\right)^{(k)}\right| \text{ uniformly bounded in } ]-x_*^\eps, x_*^\eps[
    \end{split}
    \end{equation}
    holds for any integer $k \geq 0$. Assuming uniform boundedness of $|h'(\bar\rho^\eps)^{(k)}|$ in $]-x_*^\eps, x_*^\eps[$, Lemma \ref{lem:min_explicit} yields boundedness of $(\eta^\eps)^{(k+2)}$ in $]-x_*^\eps, x_*^\eps[$. On the closed interval $[-x_*^\eps, x_*^\eps]$, we thus have $h'(\bar\rho^\eps) \in C^{k-1}([-x_*^\eps, x_*^\eps])$ and $\eta^\eps \in C^{k+1}([-x_*^\eps, x_*^\eps])$. By symmetry, it suffices to prove the claimed boundedness in \eqref{eq:deriv_implication} on the sub-interval $[0, x_*^\eps[$. By observing that $\left(x^2 - (x_*^\eps)^2\right)^2 = (x + x_*^\eps)^2\,(x-x_*^\eps)^2$ and that the derivatives of $(x + x_*^\eps)^{-2}$ are bounded on $[0, x_*^\eps[$, we can replace the denominator in \eqref{eq:deriv_implication} with $(x - x_*^\eps)^2$. Using $\eta^\eps \in C^{k+1}$, we express $\eta^\eps$ in terms of its $(k+1)$'st Taylor polynomial and insert $\eta^\eps(x_*^\eps) = \eta^\eps_x(x_*^\eps) = 0$:
    \begin{align*}
        \eta^\eps(x) = \sum_{j = 0}^{k+1} \frac{(\eta^\eps)^{(j)}(x_*^\eps)}{j!}(x - x_*^\eps)^j + R(x) = (x - x_*^\eps)^2 \sum_{j=0}^{k-1} \frac{(\eta^\eps)^{(j + 2)}(x_*^\eps)}{(j+2)!}(x - x_*^\eps)^j + R(x).
    \end{align*}
    Dividing by $(x - x_*^\eps)^2$ and taking $k$ derivatives, the polynomial term vanishes. For the remainder term, we obtain by the product rule:
    \begin{align*}
        &\left(\frac{R(x)}{(x - x_*^\eps)^2}\right)^{(k)} = \sum_{j=0}^k \binom{k}{j} (j+1)!\, \frac{R^{(k-j)}(x)}{(x - x_*^\eps)^{2+j}} \\
        &= \sum_{j=0}^k \binom{k}{j} (j+1)!\, \frac{1}{(x - x_*^\eps)^{2+j}} \left((\eta^\eps)^{(k-j)}(x) - \sum_{i = 0}^{j + 1} \frac{(\eta^\eps)^{(k - j + i)}(x_*^\eps)}{i!}(x - x_*^\eps)^{i} \right).
    \end{align*}
    Taylor's theorem applied to $(\eta^\eps)^{(k - j)}$ yields
    \begin{align*}
        \left|(\eta^\eps)^{(k-j)}(x) - \sum_{i = 0}^{j + 1} \frac{(\eta^\eps)^{(k - j + i)}(x_*^\eps)}{i!}(x - x_*^\eps)^{i} \right| \leq \frac{1}{(j+2)!} \,\max_{0 \leq \xi < x_*^\eps} |(\eta^\eps)^{(k + 2)}(\xi)| \,(x - x_*^\eps)^{2+j}.
    \end{align*}
    Since $(\eta^\eps)^{(k+2)}$ is uniformly bounded in $[0, x_*^\eps[$ as it has been shown, each summand above is uniformly bounded, proving \eqref{eq:deriv_implication}.
    
    Hence, in order to prove \eqref{eq:etaeps_deriv_estim}, it is sufficient to prove uniform boundedness of the derivatives of $h'(\bar\rho^\eps)$ up to third order. Note that Lemma \ref{lem:min_explicit} directly implies uniform boundedness of $\bar\rho^\eps$ and $\bar\rho^\eps_x$. In particular, \eqref{eq:deriv_implication} yields \eqref{eq:etaeps_deriv_estim} for $k = 0$. By Hypothesis \ref{hyp:hderiv}, we obtain
    \begin{align*}
        |h'(\bar\rho^\eps)_x| = |h''(\bar\rho^\eps)\,\bar\rho^\eps_x| \leq \frac{A}{2} \sqrt{\bar\rho^\eps} \,|\bar\rho^\eps_x|,
    \end{align*}
    which is uniformly bounded, thus \eqref{eq:deriv_implication} yields \eqref{eq:etaeps_deriv_estim} for $k = 1$.
    
    By construction of $u^\eps$, these estimates yield a uniform $L^\infty$-bound on $(\sqrt{\bar\rho^\eps})_x$. Using assumption \eqref{eq:hderbound_fourth} from Hypothesis \ref{hyp:hderiv}, we now obtain an estimate on $h'(\bar\rho^\eps)_{xxx}$: Let $\bar r$ be fixed such that $\bar\rho^\eps \leq \bar r$ for every sufficiently small $\eps > 0$. Note that integrating \eqref{eq:hderbound_fourth} with $H = \bar r$ yields
    \begin{align} \label{eq:hderiv_bound_third}
        |h'''(r)| \leq |h'''(\bar r)| + \frac{L_{\bar r}}{8} \int_r^{\bar r} s^{-3/2}\,\intd s \leq |h'''(\bar r)| + \frac{L_{\bar r}}{4\sqrt{r}} \leq \frac{\hat L_{\bar r}}{4\sqrt{r}}
    \end{align}
    for all $0 < r \leq \bar r$, where $\hat L_{\bar r} = 4\sqrt{\bar r}\,|h'''(\bar r)| + L_{\bar r}$. Together with \eqref{eq:hderbound_second} and \eqref{eq:hderbound_fourth}, this yields
    \begin{align*}
        &\left|h'(\bar\rho^\eps)_{xxx}\right| = \left|h^{(4)}(\bar\rho^\eps)\,(\bar\rho^\eps_x)^3 + h''(\bar\rho^\eps)\,\bar\rho^\eps_{xxx} + 3 h'''(\bar\rho^\eps)\,\bar\rho^\eps_x \,\bar\rho^\eps_{xx} \right| \\
        &\leq \frac{L_{\bar r}}{8} \left| \frac{(\bar\rho^\eps_x)^3}{(\bar\rho^\eps)^{3/2}}\right| + \frac{A}{2}\left|\sqrt{\bar\rho^\eps}\bar\rho^\eps_{xxx}\right| + \frac{3\hat L_{\bar r}}{4} \left|\frac{\bar\rho^\eps_x\,\bar\rho^\eps_{xx}}{\sqrt{\bar\rho^\eps}} \right| \leq \frac{L}{8}\left(\left|\left(\sqrt{\bar\rho^\eps}\right)_x\right|^3 + \left|\sqrt{\bar\rho^\eps}\bar\rho^\eps_{xxx}\right| + \left|\left(\sqrt{\bar\rho^\eps}\right)_x\,\bar\rho^\eps_{xx} \right| \right)
    \end{align*}
    with $L = \max\{L_{\bar r},\, 4A,\, 6\hat L_{\bar r}\}$.
    Uniform boundedness of $\bar\rho^\eps_{xx}$ and $\bar\rho^\eps_{xxx}$ follows from Lemma \ref{lem:min_explicit} and the estimate for $h'(\bar\rho^\eps)_x$. Since $\sqrt{\bar\rho^\eps}$ is uniformly bounded in $W^{1, \infty}$, as we have already seen, this yields the desired estimate for $h'(\bar\rho^\eps)_{xxx}$, proving \eqref{eq:etaeps_deriv_estim} using \eqref{eq:deriv_implication}.
\end{proof}

\section{Exponential convergence of solution}
\label{sec:exp_cvgce}
\subsection{Construction of the Lyapunov functional} As outlined in the introduction, we construct a uniformly convex auxiliary potential $W_\eps$ by extending the profile $-\sqrt{\bar\rho^\eps}$ to the whole real line in a specific way. Specifically, we set
\begin{align} \label{eq:weps_constr_taylor}
     W_\eps(x) := \left\{
     \begin{array}{ll}
         -c_1^\eps (x + x_*^\eps) + \frac{c_2^\eps}{2} (x + x_*^\eps)^2 & \text{ for } x \in\, ]-\infty, -x_*^\eps] \\
         -\sqrt{\bar\rho^\eps} & \text{ for } x \in [-x_*^\eps, x_*^\eps] \\
         c_1^\eps (x - x_*^\eps) + \frac{c_2^\eps}{2} (x - x_*^\eps)^2 & \text{ for } x \in [x_*^\eps, +\infty[.
     \end{array}
     \right.
 \end{align}
 with the coefficients
 \begin{align} \label{eq:weps_derivs_xstar}
     c_1^\eps := -(\sqrt{\bar\rho^\eps})_x\left((x_*^\eps)^-\right) = (\sqrt{\bar\rho^\eps})_x\left((-x_*^\eps)^+\right), \ \  c_2^\eps := -(\sqrt{\bar\rho^\eps})_{xx}\left((x_*^\eps)^-\right) = -(\sqrt{\bar\rho^\eps})_{xx}\left((-x_*^\eps)^+\right).
 \end{align}
 By the bound on the third derivative of $\sqrt{\bar\rho^\eps}$ proven in Lemma \ref{prop:min_xx_bounds}, the left and right limits in \eqref{eq:weps_derivs_xstar} exist.
 The construction is done in such a way that $W_\eps \in W^{3,\infty}_{loc}(\R)$. The second derivative $(W_\eps)_{xx}$ is continuous across $\pm x_*^\eps$ and constant outside $]-x_*^\eps, x_*^\eps[$, implying with \eqref{eq:min_xx_bound} that $W_\eps$ is indeed $\tilde\lambda_\eps$-convex on $\R$. The third derivative $(W_\eps)_{xxx}$ vanishes outside $]-x_*^\eps, x_*^\eps[$ and is bounded inside this interval by \eqref{eq:min_xxx_bound}, while possibly having jumps at $\pm x_*^\eps$.
 
 Our strategy to prove Theorem \ref{thm:main_cvgce} is to show that the auxiliary functional
\begin{align*}
    \ml_\eps(\rho) := \int_\R \left[\frac{2}{3}\rho^{3/2} + W_\eps\, \rho\right]\,\intd x
\end{align*}
decays exponentially fast along the solution constructed in section \ref{sec:sol_ex}. Note that by $\tilde\lambda_\eps$-convexity of $W_\eps$, the functional $\ml_\eps$ is geodesically $\tilde\lambda_\eps$-convex, see e.g. \cite[Theorem 5.15]{V}. We start by showing that $\bar\rho^\eps$ is indeed the minimizer of $\ml_\eps$, deriving an expression for the difference $\ml_\eps(\rho) - \ml_\eps(\bar\rho^\eps)$ along the way. As discussed in the introduction, this fact motivates the precise way we chose the potential $W_\eps$.
\begin{lemma} \label{lem:ml_min}
    The unique global minimizer of $\ml_\eps$ over $\mpt(\R)$ is given by $\bar\rho^\eps$. It holds
    \begin{equation} \label{eq:ml_diff_rep}
        \ml_\eps(\rho) - \ml_\eps(\bar\rho^\eps) = \int_\R \dst_f(\rho|\bar\rho^\eps)\,\intd x + \int_\R \rho\,(W_\eps)_+\,\intd x,
    \end{equation}
    for every $\rho \in \mpt(\R)$, where
    \begin{align*}
        \dst_f(r|\bar r) := \frac{2}{3} r^{3/2} - \frac{2}{3} \bar r^{3/2} - \sqrt{\bar r}\, (r - \bar r) \geq 0
    \end{align*}
    is the Bregman divergence of the convex function $f(r) = \frac{2}{3} r^{3/2}$.
\end{lemma}
\begin{proof}
    Using that $W_\eps = -\sqrt{\bar\rho^\eps} < 0$ inside $\{\bar\rho^\eps > 0 \}$ and $W_\eps \geq 0$ everywhere else, we have
    \begin{align*}
        \ml_\eps(\rho) - \ml_\eps(\bar\rho^\eps) &= \int_\R \left[ \frac{2}{3} \rho^{3/2} - \frac{2}{3} (\bar\rho^\eps)^{3/2} + W_\eps (\rho - \bar\rho^\eps) \right]\,\intd x \\
        &= \int_\R \left[ \frac{2}{3} \rho^{3/2} - \frac{2}{3} (\bar\rho^\eps)^{3/2} - \sqrt{\bar\rho^\eps}\, (\rho - \bar\rho^\eps) \right]\,\intd x + \int_{\{\bar\rho^\eps = 0\}} \rho\, W_\eps\,\intd x \\
        &= \int_\R \dst_f(\rho|\bar\rho^\eps)\,\intd x + \int_\R \rho\,(W_\eps)_+\,\intd x,
    \end{align*}
    proving \eqref{eq:ml_diff_rep}. As both integrals on the right-hand side are non-negative, it follows that $\bar\rho^\eps$ is a global minimizer of $\ml_\eps$, which is unique by strict geodesic convexity.
\end{proof}
 A corollary of this result is the following entropy dissipation inequality:
 \begin{corollary}
 For every $\rho \in \mpt(\R)$, it holds
     \begin{align} \label{eq:geodconv_funcineq}
    \int_\R \rho\,\left(\sqrt{\rho} + W_\eps\right)_x^2\,\intd x \geq 2\tilde\lambda_\eps \,\left(\ml_\eps(\rho)- \ml_\eps(\bar\rho^\eps)\right)
    \end{align}
\end{corollary}
\begin{proof}
    See e.g. \cite[Theorem 2.1]{CMV}.
\end{proof}

\subsection{Splitting of the equation}
In the lemma below, we transfer the formal splitting \eqref{eq:carrillo_general_intro} of equation \eqref{eq:thinfilm} to the rigorous framework provided by the minimizing movement scheme. Note that we the equation below is a version of \eqref{eq:carrillo_general_intro} that is integrated once with respect to $x$. This is needed in order to make the expression rigorous, as the potential $W_\eps$ is not expected to be four times differentiable.
\begin{lemma} \label{lem:tf_rewritten_step}
    Let $\rho \in \mpt(\R) \cap H^1(\R)$ and denote $S = \{\rho > 0\}$. If $\rho \in W^{3, \infty}_{loc}(S)$, the following holds as an equality in $L^\infty_{loc}(S)$:\begin{equation}\label{eq:tf_rewritten_step}
        \rho\, \left(-\rho_{xx} + \frac{\lambda}{2}x^2 + \eps\, h'(\rho) \right)_x = - 2\left(\rho^{3/2}(\sqrt{\rho} + W_\eps)_{xx}\right)_x + 6\, \rho\,(W_\eps)_{xx}\left(\sqrt{\rho} + W_\eps\right)_x + \eps \rho \,\mr_\eps.
    \end{equation}
    The remainder term $\mr_\eps$ is given by
     \begin{align*}
        \mr_\eps = 2\, \frac{(W_\eps)_{xxx}}{\eps}\left(\sqrt{\rho} + W_\eps\right) + v_\eps + \left(h'(\rho) - h'(\bar\rho^\eps)\right)_x,
    \end{align*}
    where $v_\eps \in C^0(\R)$ is defined as
   \begin{align*}
       v_\eps := \frac{1}{\eps} \left[-(W_\eps^2)_{xxx} + \lambda x + \eps h'(\bar\rho^\eps)_x \right].
   \end{align*}
\end{lemma}
\begin{proof}
All the derivatives above are well-defined in $S$ by assumption. Continuity of $(W_\eps^2)_{xxx}$ and thus of $v_\eps$ at $\pm x_*^\eps$ follows from the fact that $W_\eps(\pm x_*^\eps) = 0$ and boundedness of $(W_\eps)_{xxx}$.

In order to prove equation \eqref{eq:tf_rewritten_step}, we use the three elementary identities 
\begin{align*}
    \rho \,\rho_{xxx} = 2 \,\left(\rho^{3/2} (\sqrt{\rho})_{xx}\right)_x, \quad \rho \,(\sqrt{\rho})_x = \frac{1}{3} (\rho^{3/2})_x, \quad  6 (W_\eps)_x (W_\eps)_{xx} = (W_\eps^2)_{xxx} - 2 W_\eps (W_\eps)_{xxx}.
\end{align*}
Inserting these above, we obtain
\begin{align*}
    &\rho\, \left(-\rho_{xx} + \frac{\lambda}{2}x^2 + \eps\, h'(\rho) \right)_x + 2\left(\rho^{3/2}(\sqrt{\rho} + W_\eps)_{xx}\right)_x - 6\, \rho\,(W_\eps)_{xx}\left(\sqrt{\rho} + W_\eps\right)_x \\
    &=  \rho\,\lambda x + \eps \rho\, h'(\rho)_x + 2\, \rho^{3/2} (W_\eps)_{xxx} - 6\, \rho\, (W_\eps)_x (W_\eps)_{xx} \\
    &= \rho\,\lambda x + \eps \rho\, h'(\rho)_x - \rho\, (W_\eps^2)_{xxx} + 2\,\rho \,\left(\sqrt{\rho} + W_\eps\right) (W_\eps)_{xxx}.
\end{align*}
The last expression is equal to $\eps\rho\,\mr_\eps$ by definition, finishing the proof.
\end{proof}
Note that in our setting, the assumption $\rho \in W^{3, \infty}_{loc}(S)$ is natural, as it is fulfilled by the minimizer $\rho^+$ in each step of the minimizing movement scheme by Proposition \ref{prop:step_elg}.
\subsection{Proof of the functional inequalities}
As outlined in the introduction, the other key ingredient for proving a decay estimate for $\ml_\eps$ is the functional inequality \eqref{eq:funcineq_intro} that controls the remainder term $\mr_\eps$. Inequality \eqref{eq:funcineq_intro} is a direct consequence of the Proposition below and the expression for $\mr_\eps$ from Lemma \ref{lem:tf_rewritten_step}.
\begin{proposition} \label{prop:funcineqs}
    There exist constants $\kappa_{1,2,3}$ such that for all sufficiently small $\eps > 0$ and all $\rho \in \mpt(\R) \cap H^1(\R)$, the following hold:
    \begin{align}\label{eq:funcineq_veps}
\int_\R \rho\, v_\eps^2\,\intd x &\leq \kappa_1 \int_\R \rho\,\left(\sqrt{\rho} + W_\eps\right)_x^2\,\intd x ,\\ \label{eq:funcineq_wepsxxx}
\int_\R \rho\,\frac{\left|(W_\eps)_{xxx}\right|}{\eps}\left(\sqrt{\rho} + W_\eps\right)^2\,\intd x &\leq \kappa_2 \int_\R \rho\,\left(\sqrt{\rho} + W_\eps\right)_x^2\,\intd x ,\\ \label{eq:funcineq_h}
\int_\R \rho\,\left(h'(\rho) - h'(\bar\rho_\eps)\right)_x^2\,\intd x &\leq \kappa_3 \int_\R \rho\,\left(\sqrt{\rho} + W_\eps\right)_x^2\,\intd x.
\end{align}
In particular, inequality \eqref{eq:funcineq_intro} holds with $\kappa = 3\, \left(\kappa_1 + \kappa_2 + \kappa_3\right)$.
\end{proposition}
In the following, we derive the three functional inequalities \eqref{eq:funcineq_veps}-\eqref{eq:funcineq_h} separately. We start by investigating the function $v_\eps$ more closely in order to prove \eqref{eq:funcineq_veps}. 
\begin{lemma} \label{lem:veps_x_formula}
    The function $v_\eps$ is piecewise affine on $\R$. More precisely, there are constants $a_\eps \geq 0$ such that
    \begin{equation} \label{eq:veps_x_formula}
        |v_\eps(x)| = a_\eps \left(|x| - x_*^\eps\right)_+.
    \end{equation}
    Further, it holds $a_\eps \leq K$ with the constant $K$ from Lemma \ref{prop:min_xx_bounds}.
\end{lemma}
\begin{proof}
Since $W_\eps^2 = \bar\rho^\eps$ in $]-x_*^\eps, x_*^\eps[$ by construction, differentiating the Euler-Lagrange equation \eqref{eq:elg_tf} implies $v_\eps \equiv 0$ inside this interval. Outside of $]-x_*^\eps, x_*^\eps[$, it suffices to consider the interval $[x_*^\eps, +\infty[$ by symmetry. We use the fact that $(W_\eps)_{xxx}$ and $\bar\rho^\eps$ vanish to see that
\begin{align} \label{eq:veps_xx_outside}
    \eps\, v_\eps = -6\, (W_\eps)_x (W_\eps)_{xx} + \lambda x.
\end{align}
Since $W_\eps$ is equal to a quadratic polynomial in $[x_*^\eps, +\infty[$, the expression \eqref{eq:veps_xx_outside} is affine in this interval. More precisely, we have for $x_*^\eps < x < +\infty$:
\begin{align*}
    |(v_\eps)_x(x)| = \frac{1}{\eps} \left|-6\, (W_\eps)_{xx}(x_*^\eps)^2 + \lambda\right| =: a_\eps.
\end{align*}
Applying Lemma \ref{prop:min_xx_bounds} yields $a_\eps\leq \frac{K \eps}{\eps} = K$, finishing the proof.
\end{proof}
With this result, we now give a proof for the first key functional inequality \eqref{eq:funcineq_veps}.
\begin{lemma}
    Inequality \eqref{eq:funcineq_veps} holds with $\kappa_1 = \left(\frac{K}{\tilde\lambda_\eps}\right)^2$ with $\tilde \lambda_\eps$ from Lemma \ref{prop:min_xx_bounds}.
\end{lemma}
\begin{proof}
Applying Lemma \ref{lem:veps_x_formula} yields
\begin{align*}
    \int_\R \rho\, v_\eps^2\,\intd x \leq K^2 \int_\R \rho\,(|x| - x_*^\eps)_+^2\,\intd x.
\end{align*}
A comparison of the second derivatives and the fact that both $(|x| - x_*^\eps)_+^2$ and $(W_\eps)_+$ vanish inside $[-x_*^\eps, x_*^\eps]$ yield the estimate
\begin{align*}
    (|x| - x_*^\eps)_+^2 \leq \frac{2}{\tilde\lambda_\eps}(W_\eps)_+.
\end{align*}
Inserting above and applying \eqref{eq:ml_diff_rep} and \eqref{eq:geodconv_funcineq}, we obtain
\begin{align*}
    \int_\R \rho\, v_\eps^2\,\intd x \leq \frac{2K^2}{\tilde\lambda_\eps} \left(\ml_\eps(\rho) - \ml_\eps(\bar\rho^\eps)\right) \leq \frac{K^2}{\tilde\lambda_\eps^2} \int_\R \rho\,\left(\sqrt{\rho} + W_\eps\right)_x^2\,\intd x,
\end{align*}
finishing the proof.
\end{proof}

Our proof of the second functional inequality \eqref{eq:funcineq_wepsxxx} will rely on the fact that $(W_\eps)_{xxx}$ is constantly 0 outside the interval $]-x_*^\eps, x_*^\eps[$ by construction, while inside of it, it is of order $\eps$ by Lemma \ref{prop:min_xx_bounds}. These two observations directly yield the following estimate:
\begin{align*}
\int_\R \rho\,\frac{\left|(W_\eps)_{xxx}\right|}{\eps}\left(\sqrt{\rho} + W_\eps\right)^2\,\intd x \leq M \int_{-x_*^\eps}^{x_*^\eps} \rho\,(\sqrt{\rho} + W_\eps)^2\,\intd x.
\end{align*}
In order to prove \eqref{eq:funcineq_wepsxxx}, we thus need an estimate of the form
\begin{equation} \label{eq:fineq_2_simple}
    \int_{-x_*^\eps}^{x_*^\eps} \rho\,(\sqrt{\rho} + W_\eps)^2\,\intd x \leq \hat\kappa_2 \int_\R \rho\,\left(\sqrt{\rho} + W_\eps\right)_x^2\,\intd x
\end{equation}
with a constant $\hat\kappa_2$, as then \eqref{eq:funcineq_wepsxxx} will follow with $\kappa_2 = M\hat\kappa_2$. We start by proving a weaker version of \eqref{eq:fineq_2_simple}, in which the constant $\hat\kappa_2$ depends on the maximum value of $\rho$.
\begin{lemma} \label{lem:fineq_2_bound_srho}
    Let $H > 0$ be an arbitrary constant. It holds for every $\rho \in \mpt(\R) \cap H^1(\R)$:
    \begin{align*}
        \int_{[-x_*^\eps, x_*^\eps]\, \cap\, \{\rho \leq H^2\}} \rho\,(\sqrt{\rho} + W_\eps)^2\,\intd x \leq \frac{3 H}{4 \tilde\lambda_\eps} \int_\R \rho\,\left(\sqrt{\rho} + W_\eps\right)_x^2\,\intd x.
    \end{align*}
    In particular, \eqref{eq:fineq_2_simple} holds with the constant $\hat\kappa_2 = \frac{3H}{4\tilde\lambda_\eps}$ on the subset of densities $\rho$ that satisfy $\rho \leq H^2$ everywhere in $[-x_*^\eps, x_*^\eps]$.
\end{lemma}
\begin{proof}
    On $\{\rho \leq H^2\}$, it holds $\rho \leq H \sqrt{\rho}$, thus we can estimate the left-hand side:
    \begin{align*}
        &\int_{[-x_*^\eps, x_*^\eps]\, \cap\, \{\rho \leq H^2\}} \rho\,(\sqrt{\rho} + W_\eps)^2\,\intd x \leq H \int_{-x_*^\eps}^{x_*^\eps} \sqrt{\rho}\,(\sqrt{\rho} + W_\eps)^2\,\intd x \\
        &= H \int_{-x_*^\eps}^{x_*^\eps} \sqrt{\rho}\,(\sqrt{\rho} - \sqrt{\bar\rho^\eps})^2\,\intd x = H \int_{-x_*^\eps}^{x_*^\eps} \left[\rho^{3/2} + \bar\rho^\eps \sqrt{\rho} - 2 \sqrt{\bar\rho^\eps}\rho\right]\,\intd x.
    \end{align*}
    We now claim that for all real numbers $r, \bar r \geq 0$, it holds
    \begin{align} \label{eq:df_elem_ineq}
        r^{3/2} + \bar r\, \sqrt{r}-2 \sqrt{\bar r}\,r \leq \frac{3}{2} \left(\frac{2}{3} r^{3/2} - \frac{2}{3} \bar r^{3/2} - \sqrt{\bar r}\,(r - \bar r)\right) = \frac{3}{2} \dst_f(r|\bar r).
    \end{align}
    For $\bar r = 0$, this is clear. For $\bar r > 0$, we have
    \begin{align*}
        &\frac{3}{2} \left(\frac{2}{3} r^{3/2} - \frac{2}{3} \bar r^{3/2} - \sqrt{\bar r}\,(r - \bar r)\right) - \left(r^{3/2} + \bar r\, \sqrt{r}-2 \sqrt{\bar r}\,r\right) \\
        & = - \bar r^{3/2} - \frac{3}{2} \sqrt{\bar r}\,(r - \bar r) - \bar r\, \sqrt{r} + 2 \sqrt{\bar r}\,r = \frac{1}{2} \bar r^{3/2} - \bar r\,\sqrt{r} + \frac{1}{2}\sqrt{\bar r}\,r = \frac{1}{2}\bar r^{3/2}\left(1 - \sqrt{\frac{r}{\bar r}}\right)^2 \geq 0,
    \end{align*}
    proving \eqref{eq:df_elem_ineq}. This gives a pointwise estimate on the integrand above. Inserting and applying \eqref{eq:ml_diff_rep} and \eqref{eq:geodconv_funcineq} yields
    \begin{align*}
        &H \int_{-x_*^\eps}^{x_*^\eps} \left[\rho^{3/2} + \bar\rho^\eps \sqrt{\rho} - 2 \sqrt{\bar\rho^\eps}\rho\right]\,\intd x \leq \frac{3H}{2}\int_{-x_*^\eps}^{x_*^\eps} \dst_f(r|\bar r)\,\intd x \\
        &\leq \frac{3H}{2} \left(\ml_\eps(\rho) - \ml_\eps(\bar\rho^\eps)\right) \leq \frac{3H}{4\tilde\lambda_\eps} \int_\R \rho\,\left(\sqrt{\rho} + W_\eps\right)_x^2\,\intd x,
    \end{align*}
    finishing the proof.
\end{proof}
In order to obtain an estimate that holds for every $\rho$ with a uniform constant, we need to control the integral on the left-hand side of \eqref{eq:fineq_2_simple} on the set $\{\rho \geq H^2\}$. We start by proving a uniform lower bound on the right-hand side of \eqref{eq:fineq_2_simple} under the condition that this set is non-empty. We recall that by \eqref{eq:x_eps_bound}, $x_*^\eps$ is uniformly bounded away from $0$ for small $\eps$.
\begin{lemma} \label{lem:largerho_lowerbd}
    Let $H > 0$ be large enough such that $H \geq 4\, \|\bar\rho^\eps\|_{L^\infty}^{1/2}$ and $H \geq 2/\sqrt{x_*^\eps}$ for every sufficiently small $\eps > 0$. Then it holds
    \begin{align} \label{eq:largerho_lowerbd}
        \int_\R \rho\,(\sqrt{\rho} + W_\eps)_x^2\,\intd x \geq \frac{H^6}{256}
    \end{align}
    for every $\rho \in \mpt(\R) \cap H^1(\R)$ with $\{\rho \geq H^2\} \,\cap\, ]-x_*^\eps, x_*^\eps[\, \neq \emptyset$.
\end{lemma}
\begin{proof}
    By assumption, there exists $x_0 \in ]-x_*^\eps, x_*^\eps[$ such that $\rho(x_0) \geq H^2$. Without loss of generality, we assume $\rho(x_0) = H^2$ and $x_0 \leq 0$. By continuity of $\rho$, there exists some $x_1 > x_0$ with $\rho(x_1) = H^2/4$ and $\rho \geq H^2/4$ in the interval $[x_0, x_1]$. Since $\int \rho = 1$, it holds 
    \begin{align*}
        x_1 \leq x_0 + \frac{4}{H^2} \leq x_0 + x_*^\eps \leq x_*^\eps,
    \end{align*}
    where we have used that $4/H^2 \leq x_*^\eps$ by assumption. It now follows
    \begin{align*}
        &\int_\R \rho\,(\sqrt{\rho} + W_\eps)_x^2\,\intd x \geq \int_{x_0}^{x_1} \rho\,(\sqrt{\rho} + W_\eps)_x^2\,\intd x \geq \frac{H^2}{4} \int_{x_0}^{x_1} (\sqrt{\rho} + W_\eps)_x^2\,\intd x \\
        &\geq \frac{H^2}{4(x_1 - x_0)} \left( \int_{x_0}^{x_1} (\sqrt{\rho} + W_\eps)_x\,\intd x \right)^2 = \frac{H^4}{16} \left(\frac{H}{2} - H + W_\eps(x_1) - W_\eps(x_0)\right)^2 \\
        &= \frac{H^4}{16}\left(-\frac{H}{2} - \sqrt{\bar\rho^\eps(x_1)} + \sqrt{\bar\rho^\eps(x_0)}\right)^2 \geq \frac{H^4}{16}\left(\frac{H}{4}\right)^2 = \frac{H^6}{256},
    \end{align*}
    where the last estimate uses that $\sqrt{\bar\rho^\eps(x_0)} \leq H/4$ by assumption.
\end{proof}
We now prove the remaining estimate required to prove \eqref{eq:fineq_2_simple}: 
\begin{lemma} \label{lem:fineq_2_bound_lrho}
    Assume that $H > 0$ satisfies the assumptions from the previous lemma. Then it holds for every $\rho \in \mpt(\R) \cap H^1(\R)$
    \begin{align*}
        \int_{[-x_*^\eps, x_*^\eps]\, \cap\, \{\rho \geq H^2\}} \rho\,(\sqrt{\rho} + W_\eps)^2\,\intd x \leq \frac{16\, (x_*^\eps)^2}{\pi^2}\left(8 + \frac{2048 \,(c_1^\eps)^2}{H^6}\right)\int_{-x_*^\eps}^{x_*^\eps} \rho\,(\sqrt{\rho} + W_\eps)_x^2\,\intd x,
    \end{align*}
    where $c_1^\eps = (W_\eps)_x(x_*^\eps)$ is the constant from \eqref{eq:weps_derivs_xstar}.
\end{lemma}
\begin{proof}
    The left-hand side is zero in the case where $\rho \leq H^2$ everywhere inside $]-x_*^\eps, x_*^\eps[$, hence we may assume that $\{\rho \geq H^2\} \,\cap\, ]-x_*^\eps, x_*^\eps[\, \neq \emptyset$. For all $x \in ]-x_*^\eps, x_*^\eps[$ with $\rho(x) \geq H^2$, it now holds
    \begin{align*}
        |W_\eps(x)| = \sqrt{\bar\rho^\eps(x)} \leq \|\bar\rho^\eps\|_{L^\infty}^{1/2} \leq \frac{H}{4} \leq \frac{\sqrt{\rho(x)}}{4} \leq 2\sqrt{\rho(x)},
    \end{align*}
    implying that $(\sqrt{\rho(x)} + W_\eps(x))^2 \leq \sqrt{\rho(x)}^2 = \rho(x)$. Thus
    \begin{align*}
        &\int_{[-x_*^\eps, x_*^\eps]\, \cap\, \{\rho \geq H^2\}} \rho\,(\sqrt{\rho} + W_\eps)^2\,\intd x \leq \int_{[-x_*^\eps, x_*^\eps]\, \cap\, \{\rho \geq H^2\}} \rho^2\,\intd x.
    \end{align*}
    Using that $H^2 \geq 1/x_*^\eps$ by assumption and $\int \rho = 1$, we observe that on $\{\rho \geq H^2\}$, it holds
    \begin{align*}
       \frac{\rho}{2} \leq \rho - \frac{H^2}{2} \leq \rho - \frac{1}{2x_*^\eps}\int_{-x_*^\eps}^{x_*^\eps} \rho\,\intd x.
    \end{align*}
    Inserting above and applying Wirtinger's inequality yields
    \begin{align*}
        &\int_{[-x_*^\eps, x_*^\eps]\, \cap\, \{\rho \geq H^2\}} \rho^2\,\intd x \leq 4 \int_{[-x_*^\eps, x_*^\eps]\, \cap\, \{\rho \geq H^2\}} \left(\rho - \frac{1}{2x_*^\eps}\int_{-x_*^\eps}^{x_*^\eps} \rho\,\intd y\right)^2\,\intd x \\
        &\leq 4 \int_{-x_*^\eps}^{x_*^\eps} \left(\rho - \frac{1}{2x_*^\eps}\int_{-x_*^\eps}^{x_*^\eps} \rho\,\intd y\right)^2\,\intd x \leq \frac{16\,(x_*^\eps)^2}{\pi^2} \int_{-x_*^\eps}^{x_*^\eps} \rho_x^2\,\intd x.
    \end{align*}
    Using the identity $\rho_x^2 = 4\rho\,(\sqrt{\rho})_x^2$, we obtain
    \begin{align*}
        &\int_{-x_*^\eps}^{x_*^\eps} \rho_x^2\,\intd x = 4\int_{-x_*^\eps}^{x_*^\eps} \rho\,(\sqrt{\rho})_x^2\,\intd x = 4 \int_{-x_*^\eps}^{x_*^\eps} \rho\,(\sqrt{\rho} + W_\eps - W_\eps)_x^2\,\intd x \\
        &\leq 8 \int_{-x_*^\eps}^{x_*^\eps} \rho\,(\sqrt{\rho} + W_\eps)_x^2\,\intd x + 8 \int_{-x_*^\eps}^{x_*^\eps} \rho\,(W_\eps)_x^2\,\intd x \leq 8 \int_{-x_*^\eps}^{x_*^\eps} \rho\,(\sqrt{\rho} + W_\eps)_x^2\,\intd x + 8 (c_1^\eps)^2 \\
        &\leq \left(8 + \frac{2048 \,(c_1^\eps)^2}{H^6}\right)\int_{-x_*^\eps}^{x_*^\eps} \rho\,(\sqrt{\rho} + W_\eps)_x^2\,\intd x,
    \end{align*}
    where the last step uses Lemma \ref{lem:largerho_lowerbd}. This proves the claim.
\end{proof}
Combining the results from Lemma \ref{lem:fineq_2_bound_srho} and Lemma \ref{lem:fineq_2_bound_lrho} finishes the proof of \eqref{eq:fineq_2_simple} by splitting the integral on the left-hand side into one part over $\{\rho \leq H^2\}$ and one over $\{\rho \geq H^2\}$. By the considerations above, this implies \eqref{eq:funcineq_wepsxxx}.

Finally, we prove inequality \eqref{eq:funcineq_h} by using the bounds on the higher derivatives of $h$ from Hypothesis \ref{hyp:hderiv}. Introducing the function
\begin{align*}
    g: \Rnn \to \R,\quad g(r) := 2\sqrt{r}\,h''(r),
\end{align*}
we observe that these assumptions imply $g(0) = 0$, and that $g$ is globally bounded and locally Lipschitz, as by \eqref{eq:hderbound_second} and \eqref{eq:hderiv_bound_third}, we have
\begin{align} \label{eq:g_estims}
    |g(r)| \leq A \quad \text{for all } r \geq 0, \qquad |g'(r)| = 2\,\left|\sqrt{r}\, h'''(r) + \frac{h''(r)}{2\sqrt{r}}\right| \leq \frac{\hat L_H + A}{2} \quad  \text{for } 0 < r \leq H.
\end{align}
Fixing $H_0 < +\infty$ as an $\eps$-uniform upper bound on $\|\bar\rho^\eps\|_{L^\infty}$ for small $\eps$ and denoting $L := (\hat L_{3H_0} + A)/2$, we see that $g$ is $L$-Lipschitz on $[0, 3H_0]$.

We now split the left-hand side of \eqref{eq:funcineq_h} in a particular way. Using $\rho_x = 2\sqrt{\rho}\, (\sqrt{\rho})_x$, we obtain the following:
\begin{align*}
    &\int_\R \rho\,\left(h'(\rho) - h'(\bar\rho_\eps)\right)_x^2\,\intd x \\
    &= \int_\R \rho\,\left(h''(\rho)\rho_x - h''(\bar\rho^\eps)\bar\rho^\eps_x \right)^2\,\intd x = \int_\R \rho\,\left(g(\rho)(\sqrt{\rho})_x - g(\bar\rho^\eps)(\sqrt{\bar\rho^\eps})_x\right)^2\,\intd x \\
    &= \int_\R \rho\,\left(g(\rho)(\sqrt{\rho} + W_\eps)_x - g(\rho)(W_\eps + \sqrt{\bar\rho^\eps})_x + \left(g(\rho) - g(\bar\rho^\eps)\right)(\sqrt{\bar\rho^\eps})_x\right)^2\,\intd x \\
    &\leq 3\left[\int_\R \rho\,g(\rho)^2(\sqrt{\rho} + W_\eps)_x^2\,\intd x + \int_\R \rho\,g(\rho)^2(W_\eps + \sqrt{\bar\rho^\eps})_x^2\,\intd x + \int_\R \rho\, (\sqrt{\bar\rho^\eps})_x^2 \left(g(\rho) - g(\bar\rho^\eps)\right)^2\,\intd x \right]\\
    &=: 3\,\left[\mi_1 + \mi_2 + \mi_3\right].
\end{align*}
We control each of the three integrals $\mi_{1,2,3}$ separately by the right-hand side of \eqref{eq:funcineq_h}.

For $\mi_1$, it follows immediately from \eqref{eq:g_estims} that
\begin{align} \label{eq:i1_estim}
    \mi_1 \leq A^2 \int_\R \rho\,(\sqrt{\rho} + W_\eps)_x^2\,\intd x.
\end{align}
For the last integral $\mi_3$, we prove the following technical lemma, which will also be used to estimate the second integral $\mi_2$.
\begin{lemma} \label{lem:g_df_estim}
    There exists a constant $\gamma < +\infty$ such that for all $r \geq 0$ and $0 \leq \bar r \leq H_0$, it holds
    \begin{equation} \label{eq:g_df_estim}
        r\,\left(g(r) - g(\bar r)\right)^2 \leq \gamma\, \dst_f(r|\bar r).
    \end{equation}
\end{lemma}
\begin{proof}
The proof mainly uses the same ideas as \cite[Lemma 4.9]{BMZ}. We split the proof into the two cases $r \geq 3 H_0$ and $r \leq 3 H_0$.

Consider first the case $r \geq 3 H_0$. In this case, it holds $r \geq 3\bar r$ by assumption, hence
\begin{align*}
    \dst_f(r|\bar r) = \frac{2}{3}r^{3/2} + \frac{1}{3} \bar r^{3/2} - \sqrt{\bar r}\,r \geq \frac{2}{3}r^{3/2} - \frac{1}{\sqrt{3}}r^{3/2} = \frac{2-\sqrt{3}}{3} r^{3/2}.
\end{align*}
With the first estimate from \eqref{eq:g_estims} and $r \geq 3H_0$, it thus holds
\begin{align*}
    r\,\left(g(r) - g(\bar r)\right)^2 \leq 2r \left(g(r)^2 + g(\bar r)^2\right) \leq 4 A^2 r \leq \frac{4A^2}{\sqrt{3 H_0}} r^{3/2} \leq \frac{4\sqrt{3}\, A^2}{(2 - \sqrt{3})\sqrt{H_0}} \,\dst_f(r|\bar r),
\end{align*}
proving the claim in the first case.

In the case $r \leq 3H_0$, we apply the second inequality in \eqref{eq:g_estims}, which shows that $g$ is $L$-Lipschitz in the set $[0, 3H_0]$, that contains both $r$ and $\bar r$. This yields
\begin{align*}
    r\,\left(g(r) - g(\bar r)\right)^2 \leq 3 H_0 L^2 (r - \bar r)^2.
\end{align*}
For $\dst_f(r|\bar r)$, we observe that
\begin{align*}
    \dst_f(r|\bar r) &= \frac{2}{3}r^{3/2} + \frac{1}{3} \bar r^{3/2} - \sqrt{\bar r}\,r = \int_{\bar r}^r \left[\sqrt{s} - \sqrt{\bar r}\right]\,\intd s = \int_{\bar r}^r \int_{\bar r}^s \frac{1}{2\sqrt{\sigma}}\,\intd \sigma\,\intd s \\
    &\geq \frac{1}{2 \sqrt{3H_0}} \int_{\bar r}^r (s - \bar r)\,\intd s = \frac{1}{4 \sqrt{3H_0}} (r - \bar r)^2.
\end{align*}
Note that the estimate of the integral above is valid for $r \geq \bar r$ as well as for $r < \bar r$, as in the second case, the ordering of the integral bounds is reversed in both integrals. Combining these estimates yields the claim.
\end{proof}
From this result, we directly obtain an estimate for the integral $\mi_3$. Observe that $(\sqrt{\bar\rho^\eps})_x ^2\leq (c_1^\eps)^2$ in $]-x_*^\eps, x_*^\eps[$, with the constant $c_1^\eps$ from \eqref{eq:weps_derivs_xstar}. Outside this interval, $(\sqrt{\bar\rho^\eps})_x$ is constantly zero. Since $\bar\rho^\eps \leq H_0$ everywhere, it thus holds with \eqref{eq:g_df_estim}:
\begin{align*}
    \mi_3 = \int_\R \rho\, (\sqrt{\bar\rho^\eps})_x^2 \left(g(\rho) - g(\bar\rho^\eps)\right)^2\,\intd x \leq (c_1^\eps)^2\gamma \int_R \dst_f(r|\bar r)\,\intd x.
\end{align*}
Applying \eqref{eq:ml_diff_rep} and \eqref{eq:geodconv_funcineq} yields
\begin{align} \label{eq:i3_estim}
    \mi_3 \leq \frac{(c_1^\eps)^2\gamma}{2\tilde\lambda_\eps} \int_\R \rho\,(\sqrt{\rho} + W_\eps)_x^2\,\intd x,
\end{align}
which gives the desired inequality for $\mi_3$.

In order to estimate $\mi_2$, we use the construction 
\eqref{eq:weps_constr_taylor} of $W_\eps$ to obtain
\begin{align*}
    \mi_2 &=\int_\R \rho\,g(\rho)^2(W_\eps + \sqrt{\bar\rho^\eps})_x^2\,\intd x  = \int_\R \rho\,g(\rho)^2\left((W_\eps)_+\right)_x^2\,\intd x \\
    &= \int_{x_*^\eps}^{+\infty} \rho\,g(\rho)^2\left(c_1^\eps + c_2^\eps\,(x - x_*^\eps)\right)^2\,\intd x + \int_{-\infty}^{-x_*^\eps} \rho\,g(\rho)^2\left(-c_1^\eps + c_2^\eps\,(x + x_*^\eps)\right)^2\,\intd x \\
    &= \int_{x_*^\eps}^{+\infty} \rho\,g(\rho)^2\left((c_1^\eps)^2 + c_2^\eps\,\left(2 c_1^\eps (x - x_*^\eps) + c_2^\eps (x-x_*^\eps\right)\right)\,\intd x \\
    &+ \int_{-\infty}^{-x_*^\eps} \rho\,g(\rho)^2\left((c_1^\eps)^2 + c_2^\eps\,\left(-2c_1^\eps (x + x_*^\eps) + c_2^\eps (x + x_*^\eps)\right)\right)^2\,\intd x \\
    &= (c_1^\eps)^2\int_{\R \setminus [-x_*^\eps, x_*^\eps]} \rho\,g(\rho)^2\,\intd x + \int_\R \rho\,g(\rho)^2(W_\eps)_+\,\intd x.
\end{align*}
For the first integral, we use the fact that $\bar\rho^\eps = 0$ in the set $\R \setminus [-x_*^\eps, x_*^\eps]$ and that $g(0) = 0$. For the second integral, we apply the first estimate in \eqref{eq:g_estims}. With \eqref{eq:g_df_estim} and \eqref{eq:ml_diff_rep}, we obtain
\begin{align*}
    \mi_2 &\leq (c_1^\eps)^2 \int_{\R \setminus [-x_*^\eps, x_*^\eps]} \rho\,\left(g(\rho) - g(\bar\rho^\eps)\right)^2 + A^2\int_\R\rho\,(W_\eps)_+\,\intd x \\
    &\leq (c_1^\eps)^2\gamma \int_\R \dst_f(\rho|\bar\rho^\eps)\,\intd x + A^2\int_\R\rho\,(W_\eps)_+\,\intd x \\
    &\leq \max\left\{(c_1^\eps)^2\gamma, A^2\right\} \left(\ml_\eps(\rho) - \ml_\eps(\bar\rho^\eps)\right) \leq \frac{\max\left\{(c_1^\eps)^2\gamma, A^2\right\}}{2\tilde\lambda_\eps} \int_\R \rho\,(\sqrt{\rho} + W_\eps)_x^2\,\intd x.
\end{align*}
By combining the proven estimates for $\mi_1, \mi_2$ and $\mi_3$, we obtain inequality \eqref{eq:funcineq_h}, finishing the proof of Proposition \ref{prop:funcineqs}.

\subsection{Decay of the auxiliary functional along step}
We now combine the splitting from Lemma \ref{lem:tf_rewritten_step}, the functional inequality \eqref{eq:funcineq_intro} proven in the previous subsection and the results from section \ref{sec:sol_ex} to obtain a rigorous decay estimate for $\ml_\eps$ along a step of the minimizing movement scheme. In the following, we fix some $\hat\rho \in \mpt(\R) \cap H^1(\R)$ and $\rho^+$ as in Proposition \ref{prop:step_min_ex}, with the open set $S := \{\rho^+ > 0\}$. We begin with the following technical result, which is a corollary of Proposition \ref{prop:step_elg} and the geodesic convexity of $\ml_\eps$.
\begin{lemma}
    There holds the estimate
    \begin{equation} \label{eq:step_decay_abovetang}
        \frac{\ml_\eps(\hat\rho) - \ml_\eps(\rho^+)}{\tau} \geq \int_S \rho^+\,\left(\sqrt{\rho^+} + W_\eps\right)_x  \left(-\rho^+_{xx} + \frac{\lambda}{2} x^2 + \eps h'(\rho^+)\right)_x\,\intd x.
    \end{equation}
\end{lemma}
\begin{proof}
    We apply the characterization of the metric subdifferential from \cite[Theorem 10.4.13]{AGS} to the functional $\ml_\eps$. In order to check that the assumptions of the theorem are fulfilled, we observe that $(\sqrt{\rho^+})_x = \frac{\rho^+_x}{2\sqrt{\rho^+}} \in L^2(\intd \rho^+)$, since finiteness of $\me_\eps(\rho^+)$ implies $\rho^+_x \in L^2(\R)$. Moreover, since $W_\eps$ is a $\tilde\lambda_\eps$-convex outer potential and the convex function $f(r) = \frac{2}{3}r^{3/2}$ has all the needed properties, the theorem is applicable and we obtain $(\sqrt{\rho^+} + W_\eps)_x \in \partial \ml_\eps(\rho^+)$. Since $\ml_\eps$ is geodesically convex, the above-tangent formula \cite[(10.1.7)]{AGS} yields
    \begin{align*}
        \ml_\eps(\hat\rho) - \ml_\eps(\rho^+) \geq \int_S \rho^+\,\left(\sqrt{\rho^+} + W_\eps\right)_x \left(T - \mathrm{id}\right)\,\intd x.
    \end{align*}
    Inserting the Euler-Lagrange equation \eqref{eq:step_elg} proves the claim.
\end{proof}

We are now in position to prove a decay estimate for $\ml_\eps$ along the step from $\hat\rho$ to $\rho^+$.
\begin{lemma} \label{lem:step_decay}
    There exists an $\eps_0 > 0$ and a constant $C > 0$, independent of $\eps \in (0, \eps_0]$ and independent of $\hat\rho$, such that
    \begin{align} \label{eq:step_decay}
        \frac{\ml_\eps(\hat\rho) - \ml_\eps(\rho^+)}{\tau} \geq \left(2\lambda - C\eps\right) \left(\ml_\eps(\rho^+) - \ml_\eps(\bar\rho^\eps)\right)
    \end{align}
    for all $\eps \in (0, \eps_0]$ and all $\hat\rho \in \mpt(\R) \cap H^1(\R)$.
\end{lemma}
\begin{proof}
We insert \eqref{eq:tf_rewritten_step} for $\rho = \rho^+$ into \eqref{eq:step_decay_abovetang}, split the integral and obtain
\begin{equation} \label{eq:ldiff_split}
\begin{split}
    &\frac{\ml_\eps(\hat\rho) - \ml_\eps(\rho^+)}{\tau} \geq \int_S \rho^+\,\left(\sqrt{\rho^+} + W_\eps\right)_x  \left(-\rho^+_{xx} + \frac{\lambda}{2} x^2 + \eps h'(\rho^+)\right)_x\,\intd x \\
    &= -2\int_S \left(\sqrt{\rho^+} + W_\eps\right)_x \left((\rho^+)^{3/2}(\sqrt{\rho^+} + W_\eps)_{xx}\right)_x\,\intd x + 6\int_S \rho^+(W_\eps)_{xx}\left(\sqrt{\rho^+} + W_\eps\right)_x^2 \\
    &+ \eps \int_S \rho^+\, \mr_\eps\,\left(\sqrt{\rho^+} + W_\eps\right)_x\,\intd x
\end{split}
\end{equation}
Every integral above exists on its own: For the second integral, this is clear since the integrand is non-negative. For the last integral, observe that $(\sqrt{\rho^+} + W_\eps)_x \in L^2(\intd \rho^+)$ since $\rho^+ \in H^1(\R)$ and has finite second moment. Moreover, the functional inequality \eqref{eq:funcineq_intro} proves that also $\mr_\eps \in L^2(\intd \rho^+)$, hence the last integral exists. The sum over all three integrands is also absolutely integrable over $\R$, thus the first integral exists as well.

We now prove that the first term is non-negative by integrating by parts. This step is where the key idea from the introduction that the fourth-order terms can be omitted is made rigorous. Representing $S$ like in \eqref{eq:s_interval_union}, we obtain
\begin{align*}
    &-2\int_S \left(\sqrt{\rho^+} + W_\eps\right)_x \left((\rho^+)^{3/2}(\sqrt{\rho^+} + W_\eps)_{xx}\right)_x\,\intd x \\
    &= \sum_{n=1}^\infty \left[\left.\left(\sqrt{\rho^+} + W_\eps\right)_x (\rho^+)^{3/2}(\sqrt{\rho^+} + W_\eps)_{xx}\right|_{a_n}^{b_n} + \int_{a_n}^{b_n} (\rho^+)^{3/2}(\sqrt{\rho^+} + W_\eps)_{xx}^2\,\intd x\right].
\end{align*}
The integrals at the end are non-negative. We prove that the boundary terms at $b_n$ vanish, as the proof for $a_n$ is analogous. More precisely, we denote
\begin{equation*} \label{eq:bdry_vanish}
    u := \left(\sqrt{\rho^+} + W_\eps\right)_x (\rho^+)^{3/2}(\sqrt{\rho^+} + W_\eps)_{xx}
\end{equation*}
and prove that $u(x) \to 0$ along a subsequence $]a_n, b_n[\, \ni x \to b_n$.
Observe that by differentiating the square-root twice, we have
\begin{align*}
 u = \left(\sqrt{\rho^+} + W_\eps\right)_x \left(\frac{1}{2}\rho^+ \rho^+_{xx} - \frac{1}{4}(\rho_x^+)^2 + (\rho^+)^{3/2} (W_\eps)_{xx}\right).
\end{align*}
We first consider the case $b_n = +\infty$. We obtain the claim by showing that
\begin{align}\label{eq:u_intable}
    \left|\int_{a_n + 1}^{+\infty} u(x)\,\intd x \right| < +\infty.
\end{align}
To prove \eqref{eq:u_intable}, note that $\rho^+_{xx} \in L^2(\intd \rho^+)$ by the boundedness of $\rho^+$ and Proposition \ref{prop:step_reg}. Together with $(\sqrt{\rho^+} + W_\eps)_x \in L^2(\intd \rho^+)$ and the bound on $(W_\eps)_{xx}$, this implies
\begin{align*}
    \left(\sqrt{\rho^+} + W_\eps\right)_x \left(\frac{1}{2}\rho^+ \rho^+_{xx} + (\rho^+)^{3/2} (W_\eps)_{xx}\right) \in L^1(\R).
\end{align*}
Hence, in order to prove \eqref{eq:u_intable}, it remains to consider the middle term on the right. Integrating by parts yields
\begin{align*}
    \left|\int_{a_n + 1}^{+\infty}  (\rho^+_x)^2\,(W_\eps)_x\,\intd x\right| = \left|\left.\rho^+\rho^+_x \,(W_\eps)_x\, \right|_{a_n + 1}^{+\infty} - \int_{a_n + 1}^{+\infty} \rho^+\,\left(\rho^+_{xx} (W_\eps)_x + \rho^+_x (W_\eps)_{xx}\right)\,\intd x\right| < +\infty,
\end{align*}
where we have used that $(W_\eps)_x$ grows only linearly for $x \to +\infty$, and the fact that $\rho^+ \in \mpt(\R)\cap H^2(\R)$. For the remaining term, using again integration by parts implies
\begin{align*}
    \left|\int_{a_n + 1}^{+\infty}  (\rho^+_x)^2\,\left(\sqrt{\rho^+}\right)_x\,\intd x\right| = \left|\left.\sqrt{\rho^+}\,(\rho^+_x)^2\, \right|_{a_n + 1}^{+\infty} - 2\int_{a_n + 1}^{+\infty} \sqrt{\rho^+}\,\rho^+_x\,\rho^+_{xx}\,\intd x  \right| < +\infty,
\end{align*}
which follows from $\rho^+ \in H^2(\R)$ as this implies $\rho^+_{xx} \in L^2(\R)$ and $\sqrt{\rho^+},\,\rho^+_x\in L^4(\R)$. This finishes the proof of \eqref{eq:u_intable}, and thus the case $b_n = +\infty$.

In the case $b_n < +\infty$, we begin by showing that $\rho^+_{xx}$ stays bounded for $x \uparrow b_n$. Since the optimal transport map $T$ is a monotone map on $S$, the left-hand side of \eqref{eq:step_elg} is either approaches a finite value as $x \uparrow b_n$ or tends to $+\infty$. As $\lambda x + \eps h'(\rho^+)_x$ on the right-hand side are bounded near $b_n$, this implies that either $\rho^+_{xxx}$ is bounded for $x \uparrow b_n$ or it tends to $-\infty$. In the first case, it is clear that $\rho^+_{xx}$ is bounded as well. In the second case, the same result follows by observing that $\rho^+_{xx} \to -\infty$ for $x \uparrow b_n$ is impossible as $\rho^+$ is non-negative and $\rho^+_x(b_n) = 0$.

With boundedness of $\rho^+_{xx}$ near $b_n$ and the fact that $\rho^+(b_n) = \rho^+_x(b_n) = 0$, it directly follows that $|\rho^+_x|$ tends to 0 at least linearly and $\rho^+$ tends to zero at least quadratically as $x \uparrow b_n$. In conjunction with the boundedness of $(W_\eps)_{xx}$, this implies, for some constant $c$:
\begin{align*}
    \left|\frac{1}{2}\rho^+ \rho^+_{xx} - \frac{1}{4}(\rho_x^+)^2 + (\rho^+)^{3/2} (W_\eps)_{xx}\right| \leq c\,(x-b_n)^2
\end{align*}
as $x$ approaches $b_n$ from below. Since $(W_\eps)_x$ is locally bounded and $\sqrt{\rho^+}$ converges to $0$ at least linearly, which implies boundedness of $(\sqrt{\rho^+})_x$ at least along a subsequence $x \uparrow b_n$, this proves the claim, finishing the proof that the boundary terms in the integration by parts vanish and the first term in \eqref{eq:ldiff_split} is non-negative.

We now estimate the last integral in \eqref{eq:ldiff_split} using the functional inequality \eqref{eq:funcineq_intro}. Applying Young's inequality yields, with the constant $\kappa$ from \eqref{eq:funcineq_intro}:
\begin{align*}
    \int_S \rho^+\, \left|\mr_\eps\,\left(\sqrt{\rho^+} + W_\eps\right)_x\right|\,\intd x &\leq \frac{1}{2\sqrt{\kappa}} \int_S \rho^+\, \mr_\eps^2\,\intd x + \frac{\sqrt{\kappa}}{2} \int_S \rho^+\left(\sqrt{\rho^+} + W_\eps\right)_x^2\,\intd x \\
    &\leq \sqrt{\kappa} \int_S  \rho^+\left(\sqrt{\rho^+} + W_\eps\right)_x^2\,\intd x.
\end{align*}
Inserting these estimates into \eqref{eq:ldiff_split} yields with $(W_\eps)_{xx} \geq \tilde\lambda_\eps$ and \eqref{eq:geodconv_funcineq}:
\begin{align*}
    &\frac{\ml_\eps(\hat\rho) - \ml_\eps(\rho^+)}{\tau} \geq 6\int_S \rho^+(W_\eps)_{xx}\left(\sqrt{\rho^+} + W_\eps\right)_x^2 - \eps \sqrt{\kappa} \int_\R \rho^+ \left(\sqrt{\rho^+} + W_\eps\right)_x^2\,\intd x \\
    &\geq \left(6\tilde\lambda_\eps - \eps \sqrt{\kappa}\right) \int_\R \rho^+ \left(\sqrt{\rho^+} + W_\eps\right)_x^2\,\intd x \geq 2\tilde\lambda_\eps \left(6\tilde\lambda_\eps - \eps\sqrt{\kappa}\right) \left( \ml_\eps(\rho^+) - \ml_\eps(\bar\rho^\eps)\right),
\end{align*}
where the last two estimates follow use $(W_\eps)_{xx} \geq \tilde\lambda_\eps$ and \eqref{eq:geodconv_funcineq}, respectively. This proves the claim, as $12\tilde\lambda_\eps^2 \geq 2\lambda - C\eps$ for a constant $C$ by definition of $\tilde\lambda_\eps$.
\end{proof}

As a corollary of this result, we prove the first part of Theorem \ref{thm:main_cvgce}, showing uniqueness of the global minimizer $\bar\rho^\eps$ of the energy, provided that $\eps > 0$ is small enough.
\begin{corollary}
    Under the assumption that $\eps \in (0, \eps_0]$ with $\eps_0$ from the previous result, and that the constant $2\lambda - C\eps$ from \eqref{eq:step_decay} is positive, the global minimizer $\bar\rho^\eps$ of the energy $\me_\eps$ is unique.
\end{corollary}
\begin{proof}
     Let $\eps > 0$ satisfy the given assumption, and consider any global minimizer $\bar\eta \in \mpt(\R) \cap H^1(\R)$ of $\me_\eps$. Taking $\hat\rho = \bar\eta$ for a minimizing movement step, it clearly holds $\rho^+ = \bar\eta$ as well. Thus, the left-hand side of \eqref{eq:step_decay} is zero. Since $2\lambda - C\eps > 0$ by assumption, this implies $\ml_\eps(\bar\eta) \leq \ml_\eps(\bar\rho^\eps)$. As $\bar\rho^\eps$ is the unique global minimizer of the auxiliary functional $\ml_\eps$ by Lemma \ref{lem:ml_min}, it follows $\bar\eta = \bar\rho^\eps$, proving the claim.
\end{proof}

\subsection{Exponential decay along limit curve} We apply the decay estimate for $\ml_\eps$ along a single step to analyze its decay along multiple steps, and derive exponential decay along the limit curve $\rho$ from section \ref{sec:sol_ex}, proving Theorem \ref{thm:main_cvgce}. We start by deriving the following decay estimate for the time-discrete solutions $\rho^\tau$.
\begin{lemma}
    For all $\tau > 0$ and $t \geq 0$, it holds
    \begin{align} \label{eq:discrete_decay}
        \ml_\eps(\rho^\tau(t)) - \ml_\eps(\bar\rho^\eps) \leq \left(1 + (2\lambda - C\eps)\tau\right)^{-t/\tau}\left(\ml_\eps(\rho_0) - \ml_\eps(\bar\rho^\eps)\right),
    \end{align}
    with the constant $C$ from Lemma \ref{lem:step_decay}.
\end{lemma}
\begin{proof}
    We prove by induction that for every integer $n \geq 0$, it holds
    \begin{align*}
        \ml_\eps(\hat\rho_n) - \ml_\eps(\bar\rho^\eps) \leq \left(1 + (2\lambda - C\eps)\tau\right)^{-n}\left(\ml_\eps(\rho_0) - \ml_\eps(\bar\rho^\eps)\right).
    \end{align*}
    The claim is clear for $n = 0$. For $n \geq 1$, induction and \eqref{eq:step_decay} yield
    \begin{align*}
        &\ml_\eps(\hat\rho_n) - \ml_\eps(\bar\rho^\eps) = \ml_\eps(\hat\rho_n) - 
        \ml_\eps(\hat\rho_{n - 1}) + \ml_\eps(\hat\rho_{n - 1}) - \ml_\eps(\bar\rho^\eps) \\
        &\leq \ml_\eps(\rho_{n-1}^+) - \ml_\eps(\hat\rho_{n-1}) + \left(1 + (2\lambda - C\eps)\tau\right)^{-(n-1)}\left(\ml_\eps(\rho_0) - \ml_\eps(\bar\rho^\eps)\right) \\
        &\leq -(2\lambda - C\eps)\tau \,\left(\ml_\eps(\rho_{n-1}^+) - \ml_\eps(\bar\rho^\eps)\right) + \left(1 + (2\lambda - C\eps)\tau\right)^{-(n-1)}\left(\ml_\eps(\rho_0) - \ml_\eps(\bar\rho^\eps)\right)\\
        &= (2\lambda - C\eps)\tau\, \left(\ml_\eps(\hat\rho_n) - \ml_\eps(\bar\rho^\eps)\right) + \left(1 + (2\lambda - C\eps)\tau\right)^{-(n-1)}\left(\ml_\eps(\rho_0) - \ml_\eps(\bar\rho^\eps)\right).
    \end{align*}
    Adding $(2\lambda - C\eps)\tau\, (\ml_\eps(\hat\rho_n) - \ml_\eps(\bar\rho^\eps))$ to both sides and dividing by $1 + (2\lambda - C\eps)\tau$ yields the claim. To obtain \eqref{eq:discrete_decay}, let $n \geq 0$ be an integer such that $(n-1)\tau < t \leq n\tau$. Then we have $\rho^\tau(t) = \hat\rho_n$ and $-n \leq -t/\tau$, thus \eqref{eq:discrete_decay} follows from the estimate proven above.
\end{proof}
We now pass to the limit $\tau \to 0$ in \eqref{eq:discrete_decay} and obtain exponential decay of $\ml_\eps(\rho)$ along the limit curve $\rho$.
\begin{corollary}
    The weak solution $\rho$ from  Lemma \ref{lem:tau_cvgce} satisfies at every $t \geq 0$ the estimate
    \begin{align} \label{eq:continuous_exp_decay}
        \ml_\eps(\rho(t)) - \ml_\eps(\bar\rho^\eps) \leq \left(\ml_\eps(\rho_0) - \ml_\eps(\bar\rho^\eps)\right)\,e^{-(2\lambda - C \eps)t}.
    \end{align}
\end{corollary}
\begin{proof}
    The fact that right-hand side of \eqref{eq:discrete_decay} approaches the right-hand side of \eqref{eq:continuous_exp_decay} is elementary. On the left-hand side, it suffices to observe that $\ml_\eps$ is lower semi-continuous with respect to narrow convergence.
\end{proof}

In order to finish the proof of Theorem \ref{thm:main_cvgce}, we show that exponential decay of $\ml_\eps$ implies exponential convergence of $\rho$ to $\bar\rho^\eps$ in the $L^1$-norm.
\begin{corollary}
There is a constant $C_1 < +\infty$ independent of $\rho_0$ and $t$ such that for all $t \geq 0$, it holds
\begin{align*}
    \|\rho(t) - \bar\rho^\eps\|_{L^1(\R)}^2 \leq C_1\,\left(\ml_\eps(\rho_0) - \ml_\eps(\bar\rho^\eps)\right)\,e^{-(2\lambda - C \eps)t}.
\end{align*}
In particular, the solution curve $\rho$ converges to $\bar\rho^\eps$ in $L^1(\R)$ exponentially in time.
\end{corollary}
\begin{proof}
    The claim follows from a Czisz\'{a}r-Kullback type inequality: It follows from e.g. \cite[Theorem 30]{CJMTU} that there exists a constant $C_1 < +\infty$ such that
    \begin{align*}
    \|\rho - \bar\rho^\eps\|_{L^1(\R)}^2 \leq C_1 \left(\ml_\eps(\rho) - \ml_\eps(\bar\rho^\eps)\right)
    \end{align*}
    for all $\rho \in \mpt(\R)$. The claim now follows directly from the previous result.
\end{proof}

\appendix
\section{Derivative of Wasserstein distance in one space dimension}
In this section, we prove a general formula for the derivative of the Wasserstein distance between sufficiently regular probability densities in one space dimension. The result is being used in the proof of Proposition \ref{prop:step_elg}.
\begin{lemma} \label{lem:wass_deriv}
    Let $\hat\rho, \rho \in \mpt(\R)$ be probability densities on $\R$, and assume that $\rho \in C(\R)$. We fix a nonempty open interval $I =\, ]a, b[ \,\subset \R$ with $a, b \in \R \cup \{\pm \infty\}$ such that $\rho > 0$ everywhere in $I$. Additionally, we fix any $\eta \in C^\infty_c(I)$ with $\int_I \eta \,\intd x = 0$, and $s_0 > 0$ such that $\rho^s := \rho + s \eta > 0$ in $I$ for every $s \in [-s_0, s_0]$. Then, the mapping $s \mapsto \wass^2(\rho^s, \hat\rho)$ is differentiable at $s = 0$, and it holds
    \begin{align*}
        \left.\frac{\intd}{\intd s}\right|_{s = 0} \wass^2(\rho^s, \hat\rho) = 2\int_a^b \eta\,\varphi\,\intd x,
    \end{align*}
    where the function $\varphi: I \to \R$ is defined as
    \begin{align*}
        \varphi(x) := \int_{x_0}^x \left(x - T(x)\right)\,\intd x,
    \end{align*}
    where $x_0 \in I$ is arbitrary and $T \in L^\infty_{loc}(I)$ denotes the optimal transport map from $\rho$ to $\hat\rho$, restricted to $I$.
\end{lemma}
\begin{proof}
    We use the notion of (inverse) cumulative distribution functions: For a probability density $\rho \in \mpt(\R)$, its cumulative distribution function $F_\rho: \R \to [0, 1]$ is defined as
    \begin{align*}
        F_\rho(x) = \int_{-\infty}^x \rho\,\intd x.
    \end{align*}
    Moreover, its inverse cumulative distribution function $X_\rho:\, ]0, 1[\, \to \R$ is defined as
    \begin{align*}
        X_\rho(\xi) = \inf\{x \in \R: F_\rho(x) > \xi\}.
    \end{align*}
    We recall that if $\rho \in \mpt(\R)\, \cap \, C(\R)$ is a continuous density, it holds $F_\rho(X_\rho(\xi)) = \xi$ for every $\xi \in ]0, 1[$, and $X_\rho(F_\rho(x)) = x$ for every $x \in \R$ with $\rho(x) > 0$. If in addition, $\hat\rho \in \mpt(\R)$ is another density, then the unique optimal transport map $T$ that pushes $\rho$ to $\hat\rho$ is given by $T = X_{\hat\rho} \circ F_\rho$, and the Wasserstein distance between $\rho$ and $\hat\rho$ satisfies
    \begin{align} \label{eq:wass_invdistr}
        \wass^2(\rho, \hat\rho) = \int_0^1 \left(X_\rho(\xi) - X_{\hat\rho}(\xi)\right)^2\,\intd \xi.
    \end{align}

    Now, let $\hat\rho, \rho \in \mpt(\R)$, $I = \,]a, b[\,\subset \R$, $\eta \in C^\infty_c(I)$, and $s_0 > 0$ be chosen as above. Since the support of $\eta$ is compact in $I$, there exist $a_0, a_1$ with $a < a_0 < a_1 < b$ such that $\eta = 0$ everywhere outside of $[a_0, a_1]$. Denoting $\xi_0 := F_\rho(a_0)$ and $\xi_1 := F_\rho(a_1)$, the fact that $\rho$ is positive in $I$ implies $0 < \xi_0 < \xi_1 < 1$, and that $F_\rho: [a_0, a_1] \to [\xi_0, \xi_1]$ is bijective with inverse $X_\rho$. Moreover, since $\eta = 0$ outside of $[a_0, a_1]$ and $\int \eta = 0$, it holds $X_{\rho^s}(\xi) = X_\rho(\xi)$ for every $\xi \in\, ]0, 1[\, \setminus \,]\xi_0, \xi_1[$ and every $s \in [-s_0, s_0]$.

    We now claim that for every fixed $\xi \in \,]\xi_0, \xi_1[$, the map $s \mapsto X_{\rho^s}(\xi)$ is differentiable on $]-s_0, s_0[$, with derivative
    \begin{align} \label{eq:invdistr_deriv}
        \partial_s X_{\rho^s}(\xi) = -\frac{1}{\rho^s\left(X_{\rho^s}(\xi)\right)} \int_{a_0}^{X_{\rho^s}(\xi)} \eta \,\intd x.
    \end{align}
    To prove the claim, we apply the implicit function theorem to the function
    \begin{align*}
        g_\xi:\, ]-s_0, s_0[\, \times\, ]a_0, a_1[\, \to \R, \ g_\xi(s, x) = F_{\rho^s}(x) - \xi.
    \end{align*}
    It holds $g_\xi(s, x) = 0$ if and only if $x = X_{\rho^s}(\xi)$, and $g_\xi$ is continuously differentiable with
    \begin{align*}
        \partial_s g_\xi(s, x) = \int_{a_0}^x \eta\,\intd x,\quad \partial_x g_\xi(s, x) = \rho^s(x) > 0.
    \end{align*}
    Hence, the implicit function theorem is applicable and yields \eqref{eq:invdistr_deriv}.

    To prove the Lemma, we now use the form \eqref{eq:wass_invdistr} of the Wasserstein distance. Interchanging the derivative with the integral is justified, since the expression \eqref{eq:invdistr_deriv} is bounded for $s \in [-s_0, s_0]$ and $\xi \in [\xi_0, \xi_1]$. We insert \eqref{eq:invdistr_deriv} and obtain
    \begin{align*}
        \left.\frac{\intd}{\intd s}\right|_{s = 0} \wass^2(\rho^s, \hat\rho) &= \left.\frac{\intd}{\intd s}\right|_{s = 0} \int_{\xi_0}^{\xi_1} \left(X_{\rho^s}(\xi) - X_{\hat\rho}(\xi)\right)^2\,\intd \xi \\
        &= -2 \int_{\xi_0}^{\xi_1} \left(X_{\rho}(\xi) - X_{\hat\rho}(\xi)\right) \frac{1}{\rho(X_\rho(\xi))} \int_{a_0}^{X_\rho(\xi)} \eta\,\intd x\,\intd \xi.
    \end{align*}
    Substituting $\xi = F_\rho(y)$, $\intd \xi = \rho(y) \,\intd y$ yields
    \begin{align*}
        \left.\frac{\intd}{\intd s}\right|_{s = 0} \wass^2(\rho^s, \hat\rho) = -2 \int_{a_0}^{a_1} \left(y - X_{\hat\rho}(F_{\rho}(y))\right) \int_{a_0}^{y} \eta\,\intd x\,\intd y = -2\int_{a_0}^{a_1} \left(y - T(y)\right) \int_{a_0}^{y} \eta\,\intd x\,\intd y,
    \end{align*}
    where we have inserted $T = X_{\hat\rho} \circ F_\rho$.  Using the fact that $\varphi \in W^{1,\infty}(]a_0, a_1[)$ with derivative $\varphi'(y) = y - T(y)$, and integrating by parts proves the claim.
\end{proof}

\section*{Acknowledgements}

I would like to thank my supervisor Prof Daniel Matthes for all his amazing support during this research project. He introduced me to the techniques required to study fourth-order evolution equations and provided very good literature recommendations at the beginning. Not only thanks to his outstanding expertise in this subject, but also his patience and helpfulness, he was able to give great assistance, both mathematically and emotionally, throughout the entire project.
\bibliography{references}
\bibliographystyle{abbrv}

\end{document}